\documentclass[12pt]{amsart}
\usepackage{verbatim}
\usepackage[mathscr]{eucal}
\usepackage{amscd}
\usepackage{amsthm}
\usepackage{enumerate}
\usepackage{comment}
\usepackage{url}
\usepackage{color}
\usepackage{mathtools,amssymb}
\usepackage[titletoc]{appendix}
\usepackage{hyperref}
\usepackage[margin=1in]{geometry}

\newtheorem{theorem}{Theorem}[section]
\newtheorem{lemma}[theorem]{Lemma}
\newtheorem{corollary}[theorem]{Corollary}
\newtheorem{proposition}[theorem]{Proposition}

\theoremstyle{definition}
\newtheorem{example}[theorem]{Example}
\newtheorem{definition}[theorem]{Definition}
\newtheorem{remark}[theorem]{Remark}

\def\E{\mathbb{E}}

\def\cF{\mathcal{F}}

\def\e{\epsilon}

\newcommand{\xbar}{\bar{x}}

\newcommand{\calF}{\mathcal{F}}

\newcommand{\calP}{\mathcal{P}}

\newcommand{\calE}{\mathcal{E}}

\newcommand{\calS}{\mathcal{S}}

\newcommand{\emp}{\raisebox{1pt}{{$\scriptstyle\langle\makebox[.02in]{}\rangle$}}}
\newcommand{\semp}{\raisebox{0.5pt}{{$\scriptscriptstyle\langle\makebox[.01in]{}\rangle$}}}
\newcommand{\conc}{\raisebox{4pt}{{$\scriptstyle\frown$}}}
\newcommand{\sconc}{\raisebox{2pt}{{$\scriptscriptstyle\frown$}}}

\newcommand{\iseq}{\trianglelefteq}
\newcommand{\bn}[1]{{\prescript{\raisebox{1pt}{$\scriptscriptstyle <$}#1}{}{2}}}
\newcommand{\bne}[1]{{\prescript{\raisebox{1pt}{$\scriptscriptstyle \leq$}#1}{}{2}}}
\newcommand{\bl}[1]{{\prescript{#1}{}{2}}}
\newcommand{\inv}{^{\text{-}1}}

\newcommand{\seq}{\subseteq}

\newcommand{\clqedsym}{\dashv_{\text{\scriptsize claim}}}

\newcommand{\clqed}{%
  \ifmmode
    \eqno{\clqedsym}%
  \else
    \hfill$\clqedsym$%
  \fi
}

\numberwithin{equation}{theorem}

\usepackage{refcount}

\newcounter{savedequation}

\newenvironment{delayedproof}[1]
  {%
    \setcounter{savedequation}{\value{equation}}%
    \setcounter{equation}{0}%
    \renewcommand{\theequation}{\getrefnumber{#1}.\arabic{equation}}%
    \begin{proof}[\textnormal{\textbf{Proof of Theorem~\ref{#1}}}]%
  }
  {%
    \end{proof}%
    \setcounter{equation}{\value{savedequation}}%
  }

\title{Quantitative analytic stable regularity}

\author{G. Conant}
\address{Department of Mathematics, Statistics, and Computer Science\\
University of Illinois Chicago
}
\email{gconant@uic.edu}

\author{C. Terry}
\address{Department of Mathematics, Statistics, and Computer Science\\
University of Illinois Chicago
}
\email{caterry@uic.edu}

\date{July 23, 2026}

\thanks{GC was partially supported by an NSF grant (DMS-2452816); CT was partially supported by an NSF CAREER Award (DMS-2115518) and a Sloan Research Fellowship.}

\begin{document}

\begin{abstract}
 We prove quantitative stable regularity lemmas for binary real-valued functions, extending the work of Malliaris and Shelah \cite{MS} for stable graphs. The statements of our results are  modeled after non-quantitative theorems for stable functions due to Chavarria, Conant, and Pillay \cite{CPC}. One of the key tools in our quantitative proof is an ``analytic symmetry lemma", which gives a function-theoretic analogue of the fact that a pair of good sets in a graph has density close to $0$ or $1$. We also  develop a function-theoretic treatment of Malliaris and Shelah's random sampling method for refining partitions consisting of good sets into equipartitions.  
\end{abstract}

\maketitle

\section{Introduction}
 
\subsection{Summary}
The goal of this paper is to prove  stable regularity lemmas for  $[0,1]$-valued functions using finitary combinatorial techniques that yield fully quantitative and efficient bounds. In this introduction, we will first describe the body of work leading up to this goal, including  Malliaris and Shelah's \cite{MS} quantitative regularity lemma for stable graphs, and  Chavarria, Conant, and Pillay's non-quantitative regularity lemma for stable functions proved using ultraproducts. Our main results are  stated in Subsections \ref{sec:mainregstatements} and \ref{sec:maintreestatements}.

\subsection{Regularity for graphs} 
Szemer\'edi's regularity lemma is a structure theorem for arbitrary finite graphs.   Informally speaking, it says that any finite graph $G$ can be partitioned, so that between most pairs of parts in the partition, $G$ looks like a quasirandom graph. More precisely, given a finite graph $G=(V,E)$ and  nonempty sets $A,B\seq V$, let   
\[
 d_G(A,B)=\frac{|\{(x,y)\in A\times B: xy\in E\}|}{|A||B|}. 
 \]
 We say that $(A,B)$ is \emph{$\e$-regular with respect to $G$} if for all $A'\seq A$ and $B'\seq B$ with $|A'|\geq\e|A|$ and $|B'|\geq \e|B|$, $|d_G(A,B)-d_G(A',B')|\leq\e$. 

 \begin{theorem}[Szemer\'edi's regularity lemma \cite{SzemRL}]\label{thm:SRL} For all $\e>0$ there is a constant $M$ such that for any finite graph $G=(V,E)$, there is $m\leq M$ and an equipartition $V=V_1\cup\ldots\cup V_m$ such that for at least $(1-\e)m^2$ pairs $(i,j)\in [m]\times [m]$, $(V_i,V_j)$ is $\e$-regular with respect to $G$. 
 \end{theorem}

 The now standard energy increment proof of Szemer\'edi's regularity lemma shows that $M$ can be bounded by an exponential tower of height $(1/\e)^{O(1)}$.  Examples constructed by Gowers \cite{GowSRL} showed that this type of tower dependence is necessary.

\subsection{Stable regularity for graphs} 

In recent years, the regularity lemma has served as a bridge between model theory and combinatorics.  More specifically,  ``tame" versions of the regularity lemma, meaning stronger versions that hold under special combinatorial assumptions, have consistently been shown to have deep connections to model-theoretic tools; see, e.g., \cite{ChStDis,MS,MP,P}.

The first tame regularity lemmas were proved in the setting of bounded  VC-dimension by Alon, Fischer, and Newman \cite{AFN} (for bipartite graphs) and Lov\'{a}sz and Szegedy \cite{LovSzeg} (for graphons). 
Independently, stronger regularity lemmas were proved by  Malliaris and Shelah \cite{MS} for \emph{stable graphs}, which form a special subclass of the graphs of bounded VC-dimension. Stability is defined using a combinatorial configuration typically called the \emph{order property}. However, due to a certain ambiguity around this phrase in the context of functions, we will adopt Hodges' \cite{hodges} terminology of ``ladders".

\begin{definition}\label{def:ladgraphs}
Given $k\geq 1$, a \emph{$k$-ladder} for a graph $G=(V,E)$ consists of sequences $(x_1,\ldots,x_k),  (y_1,\ldots,y_k)\in V^k$ such that for all $i,j\in[k]$, if $i\leq j$ then $x_iy_j\in E$ and if $i>j$, then $x_iy_j\notin E$. When such sequences exist, we say that $G$ \emph{admits a $k$-ladder}. Otherwise, we say $G$ \emph{omits $k$-ladders}. We say $G$ is \emph{$k$-stable} if $G$ omits $k$-ladders.
\end{definition}

 The regularity lemma  for  stable graphs is stated in terms of homogeneous pairs, which we now define. Given a graph $G=(V,E)$ and nonempty sets $X,Y\seq V$, we say that the pair  $(X,Y)$ is \emph{$\e$-homogeneous in $G$} if $d_G(X,Y)<\e$ or $d_G(X,Y)> 1-\e$. It is easy to show that any $\e$-homogeneous pair is $\e^{1/3}$-regular.  

We can now state the main result of Malliaris and Shelah, which says that stable graphs admit  equipartitions with a polynomial number of parts and  homogeneity for \emph{all} pairs.
 
 \begin{theorem}[Malliaris \& Shelah \cite{MS}]\label{thm:MS}
Fix $k\geq 1$ and $0<\epsilon<1$. Then for any sufficiently large finite $k$-stable graph $G=(V,E)$, there is $m\leq O_k((1/\epsilon)^{2^{k+1}-2})$ and an equipartition $V=V_1\cup \ldots \cup V_m$ such that for all $1\leq i,j\leq m$, $(V_i,V_j)$ is $\e$-homogeneous in $G$.
 \end{theorem}
 
The proof of Theorem \ref{thm:MS} uses finitized versions of  model-theoretic tools and produces explicit polynomial bounds.\footnote{See also work of  Chernikov and Starchenko \cite{ChStNIP} and Ackerman, Freer, and Patel \cite{AFP}, which generalize the combinatorial methods developed in \cite{MS} to hypergraphs and arbitrary relational languages, respectively.}  Their techniques  serve  as the starting point for our strategy in this paper.

\subsection{Analytic regularity} 

Analogues of Szemer\'{e}di's regularity lemma for real-valued functions were initially developed by  Lov\'{a}sz and Szegedy \cite{LoSzAN} and Tao \cite{TaoSRnotes, TaoSRrevisit}. These results are sometimes referred to as \emph{analytic} regularity lemmas.   We present here a simplified form of analytic regularity which emphasizes the analogy to graph regularity. For convenience, we call a function \textit{finite} if its domain is finite. 

Given a finite function $f\colon X\times Y\rightarrow [0,1]$ and (nonempty) subsets $A\subseteq X$ and $B\subseteq Y$, we say that  $(A,B)$ is \emph{$\epsilon$-regular with respect to $f$} if  for all $A'\seq A$ and $B'\seq B$ with $|A'|\geq\e|A|$ and $|B'|\geq\e|B|$, $|\E_{A\times B}f-\E_{A'\times B'}f|\leq\epsilon$.  Note that in a graph $G=(V,E)$, $\epsilon$-regularity with respect to $\boldsymbol{1}_E$ coincides with the standard notion of $\epsilon$-regularity defined above.\footnote{Abusing notation, we view $\boldsymbol{1}_E$ as a function with domain $V\times V$.} The following is a function-theoretic generalization of Szemer\'edi's regularity lemma (Theorem \ref{thm:SRL}). See \cite[Theorem 2.6]{TaoSRnotes} and \cite[Lemma 1.1]{GrTaARL} for closely related results.
 
 \begin{theorem}[regularity for functions \cite{LoSzAN,TaoSRrevisit}]\label{thm:fnreg}
For all $\e>0$ there is a constant $M$ such that for any finite function $f\colon X\times Y\to [0,1]$, there are $ m_1,m_2\leq M$ and equipartitions $X=X_1\cup \ldots \cup X_{m_1}$ and $Y=Y_1\cup \ldots \cup Y_{m_2}$ such that for at least $(1-\e)m_1m_2$  pairs  $(i,j)\in[m_1]\times [m_2]$, $(X_i,Y_j)$ is  $\e$-regular with respect to $f$.  
 \end{theorem}

 Theorem \ref{thm:fnreg} can be deduced using the same kind of energy increment argument used to prove Theorem \ref{thm:SRL}  (see, e.g., \cite[Chapter 2]{Zhao-book}). This proof  yields the same tower-type upper bound for the dependence of $M$ on $\e$. One can see that  this type of dependence is necessary  by considering the indicator function  of a bipartite Gowers-type example from the graph setting. Such constructions exist, as shown by Fox and Lov\'{a}sz \cite{FoxLov}.

 \subsection{Qualitative analytic stable regularity}
 As the number of tame regularity lemmas has grown in the literature, a distinction has emerged between quantitative finitary proofs and non-quantitative proofs obtained via ultraproducts.  For instance, shortly after the quantitative proof of Theorem \ref{thm:MS}, Malliaris and Pillay \cite{MP} used pseudofinite methods and other tools from model-theoretic stability theory (such as definability of types) to prove a qualitative version of  Theorem \ref{thm:MS} without  explicit bounds. This proof was later revisited by Pillay \cite{P} and connected to domination by type spaces.

 In \cite{CPC}, Chavarria, Conant, and Pillay adapt methods from \cite{MP,P} to continuous logic\footnote{Here, by ``continuous logic," we mean a specific version of first-order logic for formulas taking values in the interval $[0,1]$; see \cite{BBHU}.} in order to prove a function analogue of Theorem \ref{thm:MS} (or, equivalently, a stable analogue of Theorem \ref{thm:fnreg}). To state this result, we need to define function-theoretic adaptations of ladders and homogeneous pairs.

 \begin{definition}\label{def:lad}
Given $k\geq 1$ and $\delta>0$, a \emph{$(k,\delta)$-ladder} for a function $f\colon X\times Y\to [0,1]$ consists of sequences $(x_1,\ldots,x_k)\in X^k$, $(y_1,\ldots,y_k)\in Y^k$, and $(r_1,\ldots,r_k)\in [0,1]^k$ 
such that for all $i,j\in[k]$, if $i\leq j$ then $f(x_i,y_j)\geq r_i+\delta$, and if $i>j$ then  $f(x_{i},y_{j})\leq r_i$.

 When such sequences exist, we say that $f$ \emph{admits a $(k,\delta)$-ladder}. Otherwise, we say $f$ \emph{omits $(k,\delta)$-ladders}.
\end{definition}

\begin{definition}\label{def:almostconstant}
Fix a finite function $f:X\times Y\rightarrow [0,1]$. Given $A\seq X$, $B\seq Y$, and $\delta,\epsilon> 0$, we say that the pair $(A,B)$ is \emph{$\e$-almost $\delta$-constant with respect to $f$} if there is some $E\seq A\times B$ such that $|E|> (1-\e)|A||B|$ and for any $(x,y),(x',y')\in E$,
\[
|f(x,y)-f(x',y')|<\delta.
\]
\end{definition}

Note that a $k$-ladder in a graph $G=(V,E)$ is the same thing as a $(k,\delta)$-ladder for $\boldsymbol{1}_E$, provided $\delta<1$. Similarly, a pair of subsets of $V$ is $\e$-homogeneous with respect to $G$ if and only if it is $\e$-almost $\delta$-constant with respect to $\boldsymbol{1}_E$ for some/any $\delta<1/2$. 

In analogy to graphs, one can show that in the setting of functions, an $\e$-almost $\delta$-constant pair  is $\gamma$-regular in the sense of Theorem \ref{thm:fnreg} for some $\gamma$ depending only on $\delta$ and $\epsilon$ (e.g., $\gamma=\delta+4\e^{1/3}$ suffices, though we do not claim this is optimal). However, as with graphs, being almost constant is a much stronger structural property than regularity.

We now state a non-quantitative regularity lemma for functions omitting ladders, which follows from the results in \cite{CPC}. This will include two versions: one with the usual equipartitions, and another with ``$f$-definable partitions", which are possibly non-equitable, but consist of  pieces with  low descriptive complexity in the sense of first-order definability  
(see Definition \ref{def:fdef}). 
This latter kind of partition is the main focus of \cite{CPC}, and was first made explicit in the graphs setting  by Malliaris and Pillay \cite{MP}.

\begin{theorem}[Chavarria, Conant, \& Pillay \cite{CPC}]\label{thm:CPC2}
Given $k\geq 1$ and $\delta,\epsilon>0$,  there is a constant $M$ such that the following holds.  Suppose $f\colon X\times Y\rightarrow [0,1]$ is a finite function that omits $(k,\delta)$-ladders.
\begin{enumerate}[$(a)$]
\item There are $f$-definable partitions $X=X_1\cup \ldots \cup X_{m_1}$ and $Y=Y_1\cup \ldots \cup Y_{m_2}$, with $m_1,m_2\leq M$, such that  for all $(i,j)\in[m_1]\times[m_2]$, $(X_i,Y_j)$ is $\e$-almost $(10\delta+\epsilon)$-constant with respect to $f$.
\item There are equipartitions $X=X_1\cup \ldots \cup X_{n_1}$ and $Y=Y_1\cup \ldots \cup Y_{n_2}$, with $n_1,n_2\leq M$, such that  for all $(i,j)\in[n_1]\times[n_2]$, $(X_i,Y_j)$ is $\e$-almost $(10\delta+\epsilon)$-constant with respect to $f$.
\end{enumerate}
\end{theorem}

This statement has been modified from its original version, and formatted to fit the setting of this paper. In Appendix \ref{sec:appendix}, we will provide a detailed comparison between Theorem \ref{thm:CPC2}, as stated above, and various results in \cite{CPC}. 
While the work in \cite{CPC} provides another striking connection between model-theoretic tools and regularity lemmas,  the proof methods do not yield any quantitative information about the bound $M$ in Theorem \ref{thm:CPC2}. This brings us to the primary goal of  the present article, which is a new proof of Theorem \ref{thm:CPC2} using direct combinatorial methods that provide explicit and efficient bounds. 

\begin{remark}\label{rem:fopp}
One important subtlety about Theorem \ref{thm:CPC2} is that its assumptions automatically imply a bound on ladders omitted in the ``dual" function $f^{opp}\colon Y\times X\to [0,1]$, which  sends $(y,x)$ to $f(x,y)$. In particular, if a function $f$ omits $(k,\delta)$-ladders then for any $\e>0$, $f^{opp}$ omits $(\lceil \e\inv\rceil k,\delta+\e)$-ladders (see Proposition \ref{prop:agnosticladders}$(a)$).  Since Theorem \ref{thm:CPC2} is non-quantitative, there is no reason to explicitly mention the ladder bound for $f^{opp}$ in the statement.\footnote{We also note that the actual results in \cite{CPC} are phrased in terms of omitting a different symmetric ladder configuration  (see Definition \ref{def:agnostic} and Remark \ref{rem:agnosticsym}).}  On  the other hand, since our main results are quantitative, we will use separate parameters for the ladders omitted in $f$ versus $f^{opp}$.
\end{remark}

\subsection{Quantitative analytic stable regularity} \label{sec:mainregstatements}

We now state our first main result, which is a version of Theorem \ref{thm:CPC2}$(a)$ with explicit and efficient quantitative bounds.

\begin{theorem}\label{thm:1c}
Fix $k,k'\geq 1$ and $\delta,\delta',\epsilon>0$ with $\e<1$.  Suppose $f\colon X\times Y\rightarrow [0,1]$ is a finite function such that $f$ omits $(k,\delta)$-ladders and $f^{opp}$ omits $(k',\delta')$-ladders.    Then there are $s_1\leq (9/\e)^{4^{k'}}$, $s_2\leq (9/\e)^{4^k}$ and $f$-definable  partitions $X=X_1\cup \ldots \cup X_{s_1}$ and $Y=Y_1\cup \ldots \cup Y_{s_2}$ such that for all $(i,j)\in[s_1]\times[s_2]$, $(X_i,Y_j)$ is $\epsilon$-almost $2(\delta+\delta')$-constant with respect to $f$.
\end{theorem}

In light of Remark \ref{rem:fopp}, Theorem \ref{thm:1c} directly implies Theorem \ref{thm:CPC2}$(a)$ with $m_1 \leq (9/\e)^{4^{\lceil2 /\e\rceil k}}$, $m_2\leq  (9/\e)^{4^{k}}$   and with $10\delta$ improved to $4\delta$ in the conclusion. See Theorem \ref{thm:CPC2qu} for a more detailed statement to this effect, with additional flexibility in the parameters.

The proof of Theorem \ref{thm:1c} is based on a simplified version of  Malliaris and Shelah's proof of Theorem \ref{thm:MS}, which is outlined  in the remarks following \cite[Corollary 5.10]{TW}. 
In rough terms, the proof proceeds by first partitioning $X$ and $Y$ into sets satisfying an analytic analogue of ``goodness" (as defined for graphs by Malliaris and Shelah), and then applying an analytic ``symmetry lemma" (Lemma \ref{lem:twosticks}) to conclude that pairs of good sets are almost constant. The initial partitioning step will require  an analytic version of Hodges' well-known quantitative correspondence between ladders and trees in the setting of graphs, which is proved in our companion paper \cite{CT-hodges} (see Theorem \ref{thm:hodgesfn}).\footnote{The precise form  of Theorem \ref{thm:hodgesfn} allows for a slight improvement to the  $4^k$ term in Theorem \ref{thm:1c}; see Remark \ref{rem:treetoladder}$(1)$.} In particular, Theorem \ref{thm:1c} will be preceded by an analogous result for functions omitting certain tree configurations. Further details will be given in Subsection \ref{sec:maintreestatements}.

Our second main result provides a corresponding quantitative version of Theorem \ref{thm:CPC2}$(b)$.

\begin{theorem}\label{thm:1ceq}
Fix  $k,k'\geq 1$ and $\delta,\delta',\e>0$ with $\e<1$.
Let $X$ and $Y$ be sufficiently large finite sets, and suppose $f\colon X\times Y\to [0,1]$ is  such that $f$ omits $(k,\delta)$-ladders and $f^{opp}$ omits $(k',\delta')$-ladders. Then there are $\ell_1< 62(17/\e)^{4^{k'}}$, $\ell_2< 62(17/\e)^{4^k}$, and equipartitions $X=X_1\cup \ldots \cup X_{\ell_1}$ and $Y=Y_1\cup \ldots \cup Y_{\ell_2}$
such that, for all $(i,j)\in[\ell_1]\times[\ell_2]$, $(X_i,Y_j)$ is $\e$-almost $2(\delta+\delta')$-constant with respect to $f$.
\end{theorem}

Similar to before, Theorem \ref{thm:1ceq} directly implies Theorem \ref{thm:CPC2}$(b)$ with $n_1\leq 62(17/\e)^{4^{\lceil2 /\e\rceil k}}$, $n_2\leq 62(17/\e)^{4^{k}}$ and with $10\delta$ improved to $4\delta$ in the conclusion (again, see Theorem \ref{thm:CPC2qu}). 

  To obtain  equipartitions with efficient bounds in Theorem \ref{thm:1ceq}, we will develop a function-theoretic version of the   random sampling method used in the discrete setting by Malliaris and Shelah   \cite{MS}. However, our approach will differ from that in \cite{MS} in some significant ways. In particular, Malliaris and Shelah combine random sampling with tools from VC-theory in order to prove that a vertex partition of a stable graph with good pieces can be refined into an equipartition whose pieces are still good.\footnote{Strictly speaking, Malliaris and Shelah work with a stronger notion of ``excellent sets". For the sake of simplicity, we restrict our discussion of their paper to good sets (as this notion is all  that we need here).} In the setting of functions, it turns out that the direct generalization of this approach using VC-theory only works to establish a weaker version of  Theorem \ref{thm:1ceq} due to certain subtleties that do not arise in the discrete case (see Appendix \ref{sec:appendixB} for details). For this reason, we will in fact avoid VC-theory entirely and prove a sharper equipartitioning result for arbitrary functions in terms of a weaker notion of goodness with certain quantitative advantages (see Theorem \ref{thm:equi}).

\subsection{Main results in terms of omitted trees} \label{sec:maintreestatements}

In this subsection, we reformulate Theorems \ref{thm:1c} and \ref{thm:1ceq} in the setting of functions omitting an analytic version of binary trees. This extends the discrete case of binary relations, first established by Shelah in a model-theoretic setting via ``2-rank", and reformulated later by Hodges \cite{hodges} in a combinatorial setting. The function-theoretic extension of this perspective is the main subject of our companion paper \cite{CT-hodges}. 

In order to state our notion of trees for functions, we need to review some standard notation for the tree structure on binary strings.
Given $t\geq 1$, let $\bl{t}$ denote $\{0,1\}^t$.  We view elements of $\bl{t}$ as binary strings of length $t$. By convention $\bl{0}=\{\emp\}$ contains only the empty string, denoted $\emp$.   Given $t\geq 1$, 
\[
\textstyle\bn{t}=\bigcup_{0\leq i<t}\bl{i}.
\]
We also use $\bne{t}$ to denote $\bn{t+1}$.
For $s,t\geq 0$, $\sigma \in \bl{s}$, and $\mu \in \bl{t}$, let $\sigma \conc \mu\in \bl{s+t}$ denote the concatenation of $\sigma$ and $\mu$. 
We write $\sigma\iseq \tau$ if $\sigma$ is an initial segment of $\tau$, i.e., if there is some   $\mu$ such that $\tau=\sigma\conc \mu$.

\begin{definition}\label{def:eptree1}
Given $t\geq 1$ and $\delta>0$, a \emph{$(t,\delta)$-tree} for a function $f\colon X\times Y\to [0,1]$ consists of sequences $(x_{\sigma}: \sigma\in \bn{t})$ from $X$, $(y_{\eta}: \eta\in \bl{t})$ from $Y$, and $(r_\sigma:\sigma\in \bn{t})$ from $[0,1]$
such that for all $\sigma\in \bn{t}$ and  $\eta,\eta'\in \bl{t}$, if  $\sigma\conc 0\iseq \eta$ and $\sigma\conc 1\iseq \eta'$ then
\[
f(x_{\sigma},y_{\eta})\leq r_{\sigma}\quad\text{and}\quad f(x_{\sigma},y_{\eta'})\geq r_{\sigma}+\delta.
\]
When such sequences exist, we say that $f$ \emph{admits a $(t,\delta)$-tree}. Otherwise, we say $f$ \emph{omits $(t,\delta)$-trees}.
\end{definition} 

As in the discrete case, one can establish a direct correspondence between trees and ladder coded in functions. In one direction, it is relatively straightforward to show that if $f\colon X\times Y\to [0,1]$ admits a $(2^t,\delta)$-ladder, then it also admits a $(t,\delta)$-tree (see \cite[Appendix A.2]{CT-hodges}). To prove Theorems \ref{thm:1c} and \ref{thm:1ceq}, however, we need the other direction, which requires more work. In particular, the following is one of the main results of our companion paper.

\begin{theorem}[{\cite[Theorem 1.11]{CT-hodges}}]\label{thm:hodgesfn}
Given $k\geq 1$ and $\delta>0$, if $f\colon X\times Y\to [0,1]$ admits a $(\binom{2k}{k}-1,2\delta)$-tree, then $f$ admits a $(k,\delta)$-ladder.
\end{theorem}

As we discuss in detail in \cite{CT-hodges}, related results with  weaker bounds can be derived from \cite{AndBen} and \cite{DaskGo}. Our version above is necessary to obtain polynomial bounds (in $1/\e$) in Theorems \ref{thm:1c} and \ref{thm:1ceq}. We now state the corresponding versions of these theorems in the setting of omitting trees. In many ways, one should regard these versions as the actual main results of the paper and view Theorems \ref{thm:1c} and \ref{thm:1ceq} as immediate corollaries (modulo Theorem \ref{thm:hodgesfn}; details are given below). These versions are also of inherent interest due to a direct connection between $(t,\delta)$-trees and the notion of sequential fat-shattering dimension in learning theory (see \cite[Appendix B.1]{CT-hodges}).

\begin{theorem}\label{thm:maintree}
Fix $t,t'\geq 1$ and $\delta,\delta',\e>0$ with $\e<1$.  Suppose $f\colon X\times Y\rightarrow [0,1]$ is  such that $f$ omits $(t,\delta)$-trees and $f^{opp}$ omits $(t',\delta')$-trees.  
\begin{enumerate}[$(a)$]
\item There are $s_1\leq (9/\epsilon)^{2t'}$, $s_2\leq (9/\epsilon)^{2t}$, and $f$-definable  partitions $X=X_1\cup \ldots \cup X_{s_1}$ and $Y=Y_1\cup \ldots \cup Y_{s_2}$ such that for all $(i,j)\in[s_1]\times[s_2]$, $(X_i,Y_j)$ is $\epsilon$-almost $(\delta+\delta')$-constant with respect to $f$.
\item Assume $X$ and $Y$ are sufficiently large. Then there are  $\ell_1< 62(17/\e)^{t'+1}$, $\ell_2< 62(17/\e)^{t+1}$, and equipartitions $X=X_1\cup \ldots \cup X_{\ell_1}$ and $Y=Y_1\cup \ldots \cup Y_{\ell_2}$
such that, for all $(i,j)\in[\ell_1]\times[\ell_2]$, $(X_i,Y_j)$ is $\e$-almost $(\delta+\delta')$-constant with respect to $f$.
\end{enumerate}
\end{theorem}

Part $(a)$ of Theorem \ref{thm:maintree} is proved in Section \ref{sec:mainproof1}, after which we will make several remarks on the statements. This will include discussion of the bounds, the special case of symmetric functions, and the relationship between omitting trees in $f$ versus $f^{opp}$. Part $(b)$ is proved later in Section \ref{sec:mainproof2}. 

We now explain how to deduce the main results for ladders in Subsection \ref{sec:mainregstatements} from Theorems \ref{thm:hodgesfn} and \ref{thm:maintree}.

\begin{proof}[\textnormal{\textbf{Proof of Theorem \ref{thm:1c}}}]
Theorems \ref{thm:maintree}$(a)$ and \ref{thm:hodgesfn} immediately yield the statement with $s_1\leq (9/\e)^{2t'}$ and $s_2\leq (9/\e)^{2t}$, where $t=\binom{2k}{k}-1$ and $t'=\binom{2k'}{k'}-1$. Thus we just need to check that for any $k\geq 1$,  $2\binom{2k}{k}-2\leq 4^k$. In fact, the inequality $2\binom{2k}{k}\leq 4^k$ holds, and can be proved easily by induction on $k$. 
\end{proof}

\begin{proof}[\textnormal{\textbf{Proof of Theorem \ref{thm:1ceq}}}]
 Similar to the previous argument, Theorems \ref{thm:maintree}$(b)$ and \ref{thm:hodgesfn} together imply Theorem \ref{thm:1ceq}  with $4^k$ replaced by the (smaller) value $\binom{2k}{k}$.
\end{proof}

We finish this subsection with two remarks on the previous proofs.

\begin{remark}\label{rem:treetoladder}$~$
\begin{enumerate}[$(1)$]
\item By the proof of Theorem \ref{thm:1c}, we see that in the bound on $s_1$ and $s_2$, $4^k$ can be replaced by $2\binom{2k}{k}-2$. Since $\binom{2k}{k}$ grows asymptotically like $4^k/\sqrt{\pi k}$,  this sharpens the exponent of the bound by a factor of order $\sqrt{k}$.  A similar improvement holds for Theorem \ref{thm:1ceq}. 
\item The discrepancy between $\delta$ and $2\delta$ in Theorem \ref{thm:hodgesfn} causes ``$(\delta+\delta')$-constant" in the conclusion of Theorem \ref{thm:maintree} to become ``$2(\delta+\delta')$-constant in the conclusions of Theorems \ref{thm:1c} and \ref{thm:1ceq}. It turns out that this discrepancy in Theorem \ref{thm:hodgesfn} cannot be avoided (see \cite[Proposition 3.4]{CT-hodges}). Moreover, the $\delta+\delta'$ term in Theorem \ref{thm:maintree} itself arises from our analytic symmetry lemma (Lemma \ref{lem:twosticks}), which is also tight in its dependence on $\delta$ and $\delta'$ by Example \ref{ex:twosticks}. Altogether, an improvement to either $\delta+\delta'$ in Theorem \ref{thm:maintree} or $2(\delta+\delta')$ term Theorems \ref{thm:1c} and \ref{thm:1ceq} would require an entirely different strategy. .
\end{enumerate}
\end{remark}

\subsection{Notation and conventions} \label{notation}

When working with arbitrary functions, we assume a nonempty domain. If the domain is a Cartesian product (e.g., $f\colon X\times Y\to [0,1]$), then we assume both $X$ and $Y$ are nonempty.

Throughout the paper, all sets are assumed to be finite unless otherwise stated or obviously construed from their description (e.g., the interval $[0,1]$).

We will use the following standard notation and definitions.

\begin{enumerate}[$(1)$]
\setlength{\itemsep}{5pt}
\item We use $\exp$ and $\log$ for base 2 exponentiation and logarithm.
\item Given an integer $n\geq 1$,  set $[n]:=\{1,\ldots, n\}$.  
\item Given real numbers $x,y,\delta$, we write $x\approx_{\delta} y$ to mean $|x-y|< \delta$.

\item Given a set $S$ of parameters and real numbers $x,y>0$, we write $x\leq O_S(y)$ to mean that $x\leq Cy$ for some constant $C>0$ depending only on the parameters in $S$. When $S=\emptyset$ (so $C$ is an absolute constant), we just write $x\leq O(y)$. Likewise, we write $x\geq \Omega_S(y)$ to mean that $x\geq Cy$ for some $C>0$ depending only on $S$.

\item Given sets $A\seq X$, let $\boldsymbol{1}_A\colon X\to \{0,1\}$ denote the indicator function of $A$.

\item A \emph{partition} of a set $X$ is a collection of pairwise disjoint nonempty subsets whose union is $X$. We typically write partitions as unions, e.g., $X=X_1\cup\ldots\cup X_n$. In several results, we will slightly abuse  terminology and allow partitions $X=X_0\cup X_1\cup \ldots\cup X_n$ in which the first piece $X_0$ is a small ``error set" that may be empty.  

\item An \emph{equipartition} of a finite set $X$ is a partition $X=X_1\cup\ldots\cup X_n$ such that $||X_i|-|X_j||\leq 1$ for all $1\leq i,j\leq n$.
\item \label{not:distinct} Given a set $X$ and an integer $n\geq 0$, define
\[
\binom{X}{n}=\{Y\subseteq X: |Y|=n\}\quad \text{and}\quad (X)_n=\{\xbar\in X^n:x_i\neq x_j\text{ for all }i\neq j\}.
\]
\item \label{not:levels} Given a set $E\subseteq X\times Y$ and elements $a\in X$, $b\in Y$, define
 \[
 E(a,Y) =\{y\in Y: (a,y)\in E\}\quad\text{and}\quad
 E(X,b) =\{x\in X: (x,b)\in E\}.
\]

 \item\label{not:opp} Given $f\colon X\times Y\rightarrow [0,1]$, let $f^{opp}$ denote the function from $Y\times X$ to $[0,1]$ that maps $(y,x)$ to $f(x,y)$.

\item\label{not:fibers} Given $f\colon X\times Y\rightarrow [0,1]$ and $b\in Y$, let  $f_b$ denote the function from $X$ to $[0,1]$ that maps $x$ to $f(x,b)$.  Set $\cF_f=\{f_y:y\in Y\}$, and note that $\cF_f\seq [0,1]^X$.

\item Given a finite function $f\colon X\to [0,1]$ and a nonempty set $A\seq X$, we write $\E_A f$ for $\frac{1}{|A|}\sum_{x\in A}f(x)$. We may also write $\E_A f(x)$ when emphasis on the variable is needed. Given $n\geq 1$ and $\xbar\in X^n$, we  write $\E_{\xbar}f$ for $\frac{1}{n}\sum_{i=1}^n f(x_i)$ (so, if $\xbar=(x_1,\ldots, x_n)$ has pairwise distinct coordinates, then $\E_{\xbar}f=\E_{\{x_1,\ldots, x_n\}}f$).  
\end{enumerate}

\subsection{Outline}
In Section \ref{sec:good}, we develop the notion of weakly good sets for functions and prove the analytic symmetry lemma (Lemma \ref{lem:twosticks}). In Section \ref{sec:mainproof1}, we prove Theorem \ref{thm:maintree}$(a)$. In Section \ref{sec:equipartition}, we use random sampling to prove an equipartitioning result for weakly good sets.  This is used in Section \ref{sec:mainproof2} to prove Theorem \ref{thm:maintree}$(b)$. We will also prove suitable analogues of Theorems \ref{thm:1ceq} and \ref{thm:maintree}$(b)$ for symmetric functions (see Theorems \ref{thm:1.csym} and \ref{thm:1.5sym}).

 The paper concludes with two appendices. In Appendix \ref{sec:appendix}, we compare Theorem \ref{thm:CPC2} to corresponding results in \cite{CPC}, and we explicitly derive the exact quantitative bounds for those results provided by our main theorems. In Appendix \ref{sec:appendixB}, we prove some further results related to the problem of equipartitioning good sets for functions.

\section{Good sets and an analytic symmetry lemma}\label{sec:good}

In this section, we introduce analytic generalizations of  good sets, first defined for graphs in \cite{MS}. We then prove a function-theoretic version of a symmetry lemma appearing in \cite{TW}. It will be convenient to have the following general terminology regarding ``almost constant" functions.

\begin{definition}\label{def:deltaconstant}
Fix a function $\phi\colon V\to [0,1]$ and $\delta>0$.
\begin{enumerate}[$(1)$]
\item We say $\phi$ is \emph{$\delta$-constant} if for all $v,v'\in V$, $\phi(v)\approx_{\delta}\phi(v')$. 
\item Given $W\seq V$, we say $\phi$ is \emph{$\delta$-constant on $W$} if $\phi|_W$ is $\delta$-constant.
\item Given $\e>0$, we say $\phi$ is \emph{$\e$-almost $\delta$-constant} if there is some $W\seq V$ such that $|W|>(1-\e)|V|$ and $\phi$ is $\delta$-constant on $W$.
\item Given $\e>0$ and $W\seq V$, we say $\phi$ is \emph{$\e$-almost $\delta$-constant on $W$} if $\phi|_W$ is $\e$-almost $\delta$-constant.
\end{enumerate}
\end{definition}

 We will need the following separation result for functions that are not almost constant.

\begin{proposition}\label{prop:separation}
Fix $\delta,\e>0$ and suppose $\phi\colon X\to [0,1]$ is not $2\e$-almost $\delta$-constant. Then there is some $r\in [0,1]$ such that 
\[
|\{x\in X:\phi(x)\leq r\}|\geq \e|X|\quad\text{and}\quad|\{x\in X:\phi(x)\geq  r+\delta\}|\geq\e|X|.
\]
\end{proposition}
\begin{proof}
Note that our assumptions imply $\e<\frac{1}{2}$, since any function from $X$ to $[0,1]$ is trivially $2\e'$-almost $\delta$-constant, for any $\e'>1/2$.   

 Let $N=|X|$. Without loss of generality, we may assume $X=[N]$. After relabeling, we may further assume $\phi$ is  weakly increasing. 
Set $m=\lceil \e N\rceil$, and note that $m\in [N]$ since $0<\e< \frac{1}{2}$. Let $r=\phi(m)$. Then for all $i\in [m]$, we have $\phi(i)\leq \phi(m)=r$, and hence
 \begin{equation}\label{eq:sep1bound}
 |\{i\in [N]:\phi(i)\leq r\}|\geq m\geq \e N.
 \end{equation}
 
 Now define 
 \[
 U=\{i\in [N]:r\leq \phi(i)< r+\delta\}\quad\text{and}\quad V=\{i\in [N]:\phi(i)\geq r+\delta\}.
 \] 
Note that $U$ and $V$ are disjoint. Since $\phi$ is weakly increasing and $\phi(m)=r$, we have $\{m,m+1,\ldots,N\}\seq  U\cup V$. Thus $|U|+|V|\geq N-m+1$. Also,  $|U|\leq (1-2\e)N$ since otherwise $\phi$ would be $2\e$-almost $\delta$-constant. These two inequalities yield
\begin{equation}\label{eq:sep2bound}
|V|\geq N-m+1-|U|\geq 2\e N-m+1\geq\e N,
\end{equation}
where the final inequality (equivalent to $m\leq \e N+1$) follows by definition of $m$. The result now follows from (\ref{eq:sep1bound}) and (\ref{eq:sep2bound}). 
\end{proof}

Now we define good sets for binary functions. The reader should first recall  our notation for fiber functions (item \ref{not:fibers} in Subsection \ref{notation}). 

\begin{definition}\label{def:fullgood}
Fix $\delta,\e>0$ and a function $f\colon X\times Y\to [0,1]$. We say that $A\seq X$ is \textit{$(\delta,\e)$-good with respect to $f$} if for all $y\in Y$, $f_y$ is $\e$-almost $\delta$-constant on $A$.
\end{definition}

The previous definition generalizes the notion of  $\e$-good sets in graphs, first defined in \cite{MS}. Indeed, given a graph $(V,E)$,  one can check that a subset $A\seq V$ is $\e$-good in the sense of \cite{MS} if and only if it is $(\delta,\e)$-good with respect to $\boldsymbol{1}_E$ for some/any $\delta<1$. As we will show in Section \ref{sec:mainproof1}, the significance of good sets in stable functions will be directly analogous to the graph setting. In particular, the proof of our main result on stable regularity with definable partitions (Theorem \ref{thm:1c}), will be entirely based on partitions consisting of good sets. 

On the other hand, for our main result on stable regularity with equipartitions (Theorem \ref{thm:1ceq}), we will need to work with a weaker notion of goodness in order to avoid nontrivial quantitative obstacles (which do not arise in the discrete setting and will be discussed further in Appendix \ref{sec:appendixB}).  To motivate this definition, note that a set $A\seq X$ is $(\delta,\e)$-good with respect to $f\colon X\times Y\to[0,1]$ if and only if there is a relation $E\seq A\times Y$ such that, for all $y\in Y$, $f_y$ is $\delta$-constant on $A\backslash E(A,y)$ and $|E(A,y)|<\e|A|$ (recall notation \ref{not:levels} in Subsection \ref{notation}). In other words, a good set $A\seq X$ is associated to a binary relation on $A\times Y$ all of whose $Y$-neighborhoods are sparse. The next definition weakens this by only requiring sparse neighborhoods on average.

\begin{definition}\label{def:good}
Fix $\delta,\e>0$ and $f\colon X\times Y\to [0,1]$. We say that $A\seq X$ is \textit{weakly $(\delta,\epsilon)$-good with respect to $f$} if there is some $E\seq A\times Y$ such that $|E|<\e|A||Y|$ and, for all $y\in Y$, $f_y$ is $\delta$-constant on $A\backslash E(A,y)$.
\end{definition}

The main result of this section, Lemma \ref{lem:twosticks} below, says that weakly good pairs are always almost constant (in the sense of Definition \ref{def:almostconstant}).   This is a function-theoretic extension of Lemma 5.9 from \cite{TW} (see Remark \ref{rem:TW}). We first clarify our notion of (weakly) good pairs. In a slight abuse of terminology, given $f\colon X\times Y\to [0,1]$, we will call a set $B\seq Y$ \emph{(weakly) $(\delta,\e)$-good with respect to $f$} if it is (weakly) $(\delta,\e)$-good with respect to $f^{opp}$. 

 \begin{definition}\label{def:wgpair}
Fix $\delta_1,\delta_2,\e>0$ and $f\colon X\times Y\to [0,1]$. Given $A\seq X$ and $B\seq Y$, we say that the pair $(A,B)$ is \emph{(weakly) $(\delta_1,\delta_2;\e)$-good with respect to  $f$} if $A$ is (weakly) $(\delta_1,\e)$-good with respect to $f|_{A\times B}$ and $B$ is (weakly) $(\delta_2,\e)$-good with respect to $f|_{A\times B}$.
\end{definition}

We now prove the main result of this section. 

\begin{lemma}[Analytic Symmetry Lemma]\label{lem:twosticks}
Fix $\delta_1,\delta_2,\e>0$ and $f\colon X\times Y\to [0,1]$. Suppose $A\seq X$ and $B\seq Y$ are such that $(A,B)$ is weakly $(\delta_1,\delta_2;\e)$-good with respect to $f$.  Then $(A,B)$ is  $4\e$-almost $(\delta_1+\delta_2)$-constant with respect to $f$. 
\end{lemma}
\begin{proof}
Note that we may assume $\e\leq  \frac{1}{4}$ since otherwise the conclusion holds trivially. 

Fix $E\seq A\times B$ witnessing that $A$ is weakly $(\delta_1,\e)$-good with respect to $f|_{A\times B}$, and fix $F\seq A\times B$ witnessing that $B$ is weakly $(\delta_2,\e)$-good with respect to $f|_{A\times B}$. Define $Z=E\cup F$. Since  $|E|,|F|<\e|A||B|$, we have
\begin{equation}\label{eq:Ebound}
|Z| <2\e|A||B|.
\end{equation}
 Set $\alpha=|Z|/|A||B|$ and define
\[
\e^*=\frac{2\e-\frac{\alpha}{2}}{1-\alpha}.
\]
Note that $\e^*>0$ is well-defined since $\alpha<2\e$ by (\ref{eq:Ebound}) and $2\e\leq 1$. Set $\delta^*=\delta_1+\delta_2$.

Let $D=(A\times B)\backslash Z$. 
We claim it suffices to show that $f$ is $2\e^*$-almost $\delta^*$-constant on $D$. Indeed, if $f$ is $2\e^*$-almost $\delta^*$-constant on $D$, then there is a set $U\subseteq D$ satisfying $|U|<2\e^*|D|$ such that $f$ is $\delta^*$-constant on $D\setminus U$.  This implies $f$ is $4\e$-almost $\delta^*$-constant (our desired conclusion) because $A\times B=D\cup (Z\cup U)$ and 
\[
|Z\cup U|<\alpha|A||B|+2\e^*|D|=\alpha|A||B|+2\e^*(1-\alpha)|A||B|=4\e|A||B|,
\]
where the final equality is by definition of $\e^*$.  Therefore, the rest of the proof will be devoted to proving $f$ is $2\e^*$-almost $\delta^*$-constant on $D$.

Toward a contradiction, suppose $f$ is not $2\e^*$-almost $\delta^*$-constant on $D$. By Proposition \ref{prop:separation}, there is some $r\in [0,1]$ such that, setting
\[
R_0=\{(x,y)\in D:f(x,y)\leq r\}\quad\text{and}\quad R_1=\{(x,y)\in D:f(x,y)\geq r+\delta^*\},
\]
we have 
\begin{equation}\label{eq:PQlowers}
|R_i|\geq \e^*|D|\text{ for $i\in\{0,1\}$.}
\end{equation}

Now, for $i\in\{0,1\}$, define
\[
V_i=\{x\in A:R_i(x,B)\neq\emptyset\}\quad\text{and}\quad W_i=\{y\in B:R_i(A,y)\neq\emptyset\}.
\]

\noindent\textit{Claim 1.} $V_0\cap V_1=\emptyset$ and $W_0\cap W_1=\emptyset$. 

\noindent\textit{Proof.} We prove $W_0\cap W_1=\emptyset$. The proof that $V_0\cap V_1= \emptyset$ is symmetric. 

Suppose $y\in W_0\cap W_1$. Then there are $x_0,x_1\in A$ such that $(x_0,y)\in R_0$ and $(x_1,y)\in R_1$, which implies   $f(x_0,y)\leq r$ and $f(x_1,y)\geq r+\delta^*$. On the other hand, since $R_0,R_1\seq D=(A\times B)\backslash Z$, we have $(x_0,y),(x_1,y)\not\in Z$, hence $(x_0,y),(x_1,y)\not\in E$. Since $f_y$ is $\delta_1$-constant on $A\backslash E(A,y)$, this implies $f(x_0,y)\approx_{\delta_1} f(x_1,y)$, which is a contradiction.\clqed\medskip

\noindent\textit{Claim 2.} $V_0\times W_1\seq Z$ and $V_1\times W_0\seq Z$. 

\noindent\textit{Proof.} We prove $V_0\times W_1\seq Z$. The proof that $V_1\times W_0\seq Z$ is symmetric.  

Fix $(x,y)\in V_0\times W_1$ and, toward a contradiction, suppose $(x,y)\not\in Z$. Since $x\in V_0$ and $y\in W_1$, there are $y_0\in B$ and $x_1\in A$ such that $(x,y_0)\in R_0$ and $(x_1,y)\in R_1$, hence $f(x,y_0)\leq r$ and $f(x_1,y)\geq r+\delta^*$.

On the other hand, since $R_0,R_1\seq D=(A\times B)\backslash Z$, we have $(x,y),(x,y_0),(x_1,y)\not\in Z$. As in the proof of Claim 1, this implies $f(x,y)\approx_{\delta_1} f(x_1,y)$ and $f(x,y)\approx_{\delta_2} f(x,y_0)$. By the triangle inequality, $f(x,y_0)\approx_{\delta^*}f(x_1,y)$, which is a contradiction.\clqed\medskip

Finally, for each $i\in\{0,1\}$, set $v_i=|V_i|$ and $w_i=|W_i|$. 
Combining Claims 1 and 2,  the sets $V_0\times W_1$ and $V_1\times W_0$ are disjoint subsets of $Z$, and hence
\begin{equation}\label{eq:Zlowerbound}
|Z|\geq v_0w_1+v_1w_0\geq 2\sqrt{v_0w_0v_1w_1},
\end{equation}
where the last inequality is immediate from the AM-GM inequality.

On the other hand, note that $R_i\seq V_i\times W_i$ for each $i\in\{0,1\}$. Hence (\ref{eq:PQlowers}) implies
\begin{equation}\label{eq:VtWt}
v_0w_0v_1w_1\geq |R_0||R_1|\geq (\e^*|D|)^2=((2\e-{\textstyle\frac{\alpha}{2}})|A||B|)^2,
\end{equation}
where the equality uses $|D|=(1-\alpha)|A||B|$ and the definition of $\e^*$. 
 Now (\ref{eq:Zlowerbound}) and (\ref{eq:VtWt}) together imply $|Z|\geq (4\e-\alpha)|A||B|$. With (\ref{eq:Ebound}) this implies $4\e-\alpha<2\e$, which contradicts $\alpha<2\e$.
\end{proof}

The following example shows that the term $\delta_1+\delta_2$  in Lemma \ref{lem:twosticks} cannot be improved in general.  The example will be a symmetric function of the form $f\colon X\times X\rightarrow [0,1]$. Thus we note that in this case, $X$ is (weakly) $(\delta,\e)$-good with respect to $f$ if and only if $X$ is (weakly) $(\delta,\e)$-good with respect to $f^{opp}$.

\begin{example}\label{ex:twosticks}
Fix $0<\delta\leq\frac{1}{4}$ and $0<\gamma<\delta$. Given an even integer $n\geq 2$, we construct a symmetric function $f\colon [n]\times [n]\to [0,1]$ such that $[n]$ is   $(\delta,\e)$-good with respect to $f$ for any $\epsilon>0$, but $f$ is not $\e$-almost $2\gamma$-constant when $\e<\frac{1}{4}$. In particular, define $f\colon [n]\times [n]\to [0,1]$ such that 
\[
f(x,y)=\gamma((x~\textnormal{mod}~2)+(y~\textnormal{mod}~2)).
\]
Given $i\in[n]$, if $i$ is even then the fiber $f_i$ takes values in $\{0,\gamma\}$, while if $i$ is odd then $f_i$ takes values in $\{\gamma,2\gamma\}$. Hence all fibers are $\delta$-constant, which means that for any $\e>0$, $[n]$ is $(\delta,\e)$-good with respect to $f$. Now suppose $\epsilon<\frac{1}{4}$ and fix $E\seq [n]\times [n]$ with $|E|>(1-\e)n^2$. Since $(1-\e)n^2>\frac{3}{4}n^2$, $E$ must contain points $(x,y)$ and $(x',y')$ where $x$ and $y$ are both even and $x'$ and $y'$ are both odd. So $f(x,y)=0$ and $f(x',y')=2\gamma$, which shows that $f$ is not $2\gamma$-constant on $E$.
\end{example}

\begin{remark}\label{rem:TW}
Lemma \ref{lem:twosticks} generalizes and strengthens \cite[Lemma 5.9]{TW}, which provides a related symmetry lemma for ``almost good" sets in  bipartite graphs. This notion extends to functions by defining  $A\seq X$ to be $\gamma$-almost $(\delta,\e)$-good for $f\colon X\times Y\to [0,1]$, if there is some $Y'\seq Y$ such that $|Y'|\geq (1-\gamma)|Y|$ and $A$ is $(\delta,\e)$-good with respect to $f|_{X\times Y'}$. Under this translation, \cite[Lemma 5.9]{TW} shows that, given $E\seq X\times Y$, if   $X$ and $Y$ are both $\e$-almost $(1,\e)$-good with respect to $\boldsymbol{1}_E$, then the pair $(X,Y)$ is $2\sqrt{\e}$-homogeneous with respect to $E$ (i.e.,  $\boldsymbol{1}_E$ is $2\sqrt{\e}$-almost $1$-constant). It is easy to show that a $\gamma$-almost $(\delta,\e)$-good set (for an arbitrary function) is weakly $(\delta,\gamma+\e)$-good. Therefore Lemma \ref{lem:twosticks} extends \cite[Lemma 5.9]{TW} to functions and improves the $2\sqrt{\e}$ term in the latter result to $8\e$.  

Using a standard Markov-type averaging argument, one can also show that a weakly $(\delta,\e)$-good set is $\gamma$-almost $(\delta,\e')$-good whenever $\gamma\e'=\e$. Thus weakly good sets are qualitatively equivalent to almost good sets. However, weakly good sets have quantitative advantages leading to better bounds in our results.
\end{remark}

\section{Stable regularity with definable partitions}\label{sec:mainproof1}

In this section we prove Theorem \ref{thm:maintree}$(a)$ (regularity with definable partitions for functions omitting trees). Given $f\colon X\times Y\to [0,1]$, the rough strategy is to show that if $f$ and $f^{opp}$ omit trees of bounded length, then $X$ and $Y$ can be partitioned into good sets, which then yields almost constant pairs via the symmetry lemma (Lemma \ref{lem:twosticks}). The first step is to show that one can find large good sets in functions omitting trees. This is also where ``definability" of the sets in the partition becomes relevant.  So we now state the definition of $f$-definability, following  \cite{CPC}.

\begin{definition}\label{def:fdef}
Given $f:X\times Y\rightarrow [0,1]$, we say a subset $X'\subseteq X$ is \emph{$f$-definable} if it is a (finite) Boolean combination of sets of the form
\[
\{x\in X: f(x,b)\geq r\}\quad\text{or}\quad \{x\in X: f(x,b)\leq r\}
\]
where $b\in Y$ and $r\in[0,1]$. An \emph{$f$-definable} subset of $Y$ is defined analogously.
A partition of $X$ or $Y$ is \emph{$f$-definable} if all the sets in the partition are $f$-definable.
\end{definition}

 It is  worth pointing out that this notion does not directly correspond to definable sets in continuous logic, but is instead closer to what Anderson \cite{anderson} calls ``constructible sets" in the context of continuous logic. We should also note that  definability of the partition in the stable regularity lemma for graphs (Theorem \ref{thm:MS}) is not explicitly discussed by Malliaris and Shelah \cite{MS}. However, it was later shown by Chernikov and Starchenko \cite{ChStNIP} that the ``excellent" sets  in \cite[Claim 5.4]{MS} can be found definably via a construction with $\e$-approximations (see \cite[Lemma 4.10]{ChStNIP}). We will sidestep this issue by working only with our function-theoretic analogue of good sets, where definability arises  automatically.
It is also important to note that the notion of $f$-definability becomes more robust (especially in the finite setting) if one  bounds the ``complexity" of the Boolean combination. The constructive nature of our proofs allows for this, but we have not made it explicit in the statements of results in order to avoid defining a precise notion of complexity. Instead, some brief comments are given in Remark \ref{rem:complexity} of the appendix. 

We are now ready to prove that functions omitting trees contain large good sets. To streamline the argument, we will use the following notion of rank.

\begin{definition}
Given  $f\colon X\times Y\rightarrow [0,1]$, $\delta>0$, and $B\subseteq Y$, the \emph{$\delta$-rank of  $B$ with respect to $f$} is the maximum integer $t\geq 1$ such that $f|_{X\times B}$ admits a $(t,\delta)$-tree. If no such $t$ exists, then the $\delta$-rank of $B$ is $0$. 
\end{definition}

\begin{lemma}\label{lem:findgood}
Fix $t\geq 0$ and $\e,\delta>0$ with $\e<1$. Let $f\colon X\times Y\to [0,1]$ be a function, and suppose $Y'\subseteq Y$ is a nonempty set with $\delta$-rank at most $t$.   Then there is a subset $Y''\subseteq Y'$ which is $(\delta,2\e)$-good with respect to $f$ and has size at least  $\epsilon^t|Y'|$.  Moreover,  if $Y'$ is $f$-definable, so is $Y''$.    
\end{lemma}
\begin{proof}
We prove this by induction on $t\geq 0$.

\textit{Base Case:} Suppose $Y'\subseteq Y$ has  $\delta$-rank $0$, i.e., $f|_{X\times Y'}$ omits $(1,\delta)$-trees. We show that for any $a\in X$, $f^{opp}_a|_{Y'}$ is $\delta$-constant. In particular, this will imply $Y'$ is $(\delta,2\e)$-good and hence we may take $Y''=Y'$. Toward a contradiction, suppose there are $a\in X$ and $y_0,y_1\in Y'$ such that $|f(a,y_0)-f(a,y_1)|\geq\delta$. Without loss of generality, let us assume $f(a,y_0)<f(a,y_1)$. We  now have a $(1,\delta)$-tree for $f|_{X\times Y'}$ using $x_{\semp}=a$, $y_0$, $y_1$, and $r_{\semp}=f(a,y_0)$, which is a contradiction.

\textit{Induction Step:} Fix $t\geq 1$ and assume  the claim holds for $t-1$. Suppose $Y'\subseteq Y$ has $\delta$-rank at most $t$. If $Y'$ is itself $(\delta,2\e)$-good, then taking $Y''=Y'$, we are done.  So assume now $Y'$ is not $(\delta,2\e)$-good. By definition, this means  there is some $x\in X$ such that $f^{opp}_x|_{Y'}$ is not $2\e$-almost $\delta$-constant. By Proposition \ref{prop:separation}, there is some $r\in [0,1]$ such that both 
$$
Y_0=\{y\in Y': f(x,y)\leq r\}\text{ and }Y_1=\{y\in Y': f(x,y)\geq r+\delta\}
$$
have size at least $\epsilon|Y'|$. 

We claim that  some $Z\in\{Y_0,Y_1\}$ has $\delta$-rank at most $t-1$. Indeed, if not, then for each $i\in \{0,1\}$, $f|_{X\times Y_i}$ admits a $(t,\delta)$-tree, say given by sequences
\[
(x^i_\sigma:\sigma\in \bn{t})\text{ in } X,\quad (y^i_\eta:\eta\in \bl{t})\text{ in } Y_i,\quad\text{and}\quad(r^i_\sigma:\sigma\in \bn{t})\text{ in }[0,1].
\]
By choice of $Y_0$ and $Y_1$, this yields a $(t+1,\delta)$-tree for $f|_{X\times (Y_0\cup Y_1)}$ consisting of the following sequences:
\begin{enumerate}[\hspace{5pt}$\ast$]
\item set $x_{\semp}=x$ and, for $i\in\{0,1\}$ and $\sigma\in \bn{t}$, set $x_{i\sconc\sigma}=x^i_\sigma$, 
\item for $i\in\{0,1\}$ and $\eta\in \bl{t}$, set $y_{i\sconc\eta}=y^i_\eta$, and  
\item set $r_{\semp}=r$ and, for $i\in\{0,1\}$ and $\sigma\in \bn{t}$, set $r_{i\sconc\sigma}=r^i_\sigma$.
\end{enumerate}
Since $Y_0\cup Y_1\seq Y'$, this contradicts the assumption that $Y'$ has $\delta$-rank at most $t$.

Now fix $Z\in\{Y_0,Y_1\}$ with $\delta$-rank at most $t-1$.   By induction, there is a $(\delta,2\e)$-good set $Y''\subseteq Z$ with $|Y''|\geq \e^{t-1}|Z|\geq \e^t|Y'|$, and such that if $Z$ is $f$-definable then so is $Y''$.  Note that if $Y'$ is $f$-definable, then so is $Z$, and thus so is $Y''$. This finishes the proof.
\end{proof}

We can now iterate Lemma \ref{lem:findgood} in the case where $f$ admits no $(t,\delta)$-tree to obtain a partition of most of $Y$ into good $f$-definable sets.

\begin{corollary}\label{cor:treeconstant}
Fix $t\geq 1$, and $\delta,\epsilon>0$ with  $\epsilon<1$. Suppose $f\colon X\times Y\to [0,1]$ omits $(t,\delta)$-trees.  
Then for some $s\leq \e^{-2t}$ there is a partition $Y=Y_0\cup Y_1\cup \ldots \cup Y_s$ into $f$-definable sets such that $|Y_0|< \e^2 |Y_1|$ and, for each $1\leq i\leq s$, $Y_i$ is $(\delta, 2\e)$-good with respect to $f$ and has size at least $\e^{2t}|Y|$.
\end{corollary}
\begin{proof}
We inductively construct $Y_1,\ldots,Y_s$ as follows.

\textit{Step 1:} Note $Y$ is $f$-definable and has $\delta$-rank at most $t-1$.  Apply Lemma \ref{lem:findgood} to find an $f$-definable set $Y_1\subseteq Y$ of size at least $\e^{t-1}|Y|$ which is $(\delta,2\e)$-good with respect to $f$.    

\textit{Step $m+1$:} Suppose that we have constructed pairwise disjoint $f$-definable subsets $Y_1,\ldots, Y_m$  of $Y$, such that $|Y_1|\geq \e^{t-1}|Y|$ and, for each $1\leq i\leq m$, $Y_{i}$ is $(\delta,2\e)$-good and has  size at least $\e^{2t}|Y|$. Set $Y'=Y\backslash (Y_1\cup\ldots\cup Y_m)$.  If $|Y'|<\e^2|Y_1|$, set $s=m$, $Y_0=Y'$ and end the construction.  Otherwise, note $Y'$ has $\delta$-rank at most $t-1$ and is $f$-definable because $Y_1,\ldots,Y_m$ are $f$-definable. Consequently, by  Lemma \ref{lem:findgood} there is an $f$-definable set $Y_{m+1}\subseteq Y'$ which is $(\delta,2\e)$-good and of size at least $\e^{t-1}|Y'|\geq \e^{t+1}|Y_1|\geq \e^{2t}|Y|$, where the last inequality uses $|Y_1|\geq \e^{t-1}|Y|$.   

Clearly this process will stop after some $s\leq \e^{-2t}$ steps, which finishes the proof. 
\end{proof}

For later purposes, we state a slight reformulation of the previous corollary, which follows using a nearly identical proof.

\begin{corollary}\label{cor:treeconstant2}
Fix $t\geq 1$, and $\delta,\epsilon>0$ with  $\epsilon<1$. Suppose $f\colon X\times Y\to [0,1]$ omits $(t,\delta)$-trees. 
Then for some $s\leq \e^{-t}$ there is a partition $Y=Y_0\cup Y_1\cup \ldots \cup Y_s$ into $f$-definable sets such that $|Y_0|< \epsilon |Y|$, and such that for each $1\leq i\leq s$, $Y_i$ is $(\delta, 2\e)$-good with respect to $f$ and has size at least $\e^{t}|Y|$.
\end{corollary}

We now use Corollary \ref{cor:treeconstant}, along with the analytic symmetry lemma (Lemma \ref{lem:twosticks}), to prove Theorem \ref{thm:maintree}$(a)$.

\begin{proof}[\textnormal{\textbf{Proof of Theorem \ref{thm:maintree}$(a)$}}]
Fix $t,t'\geq 1$ and $\delta,\delta',\e>0$ with $\e<1$.  Suppose $f\colon X\times Y\rightarrow [0,1]$ is  such that $f$ omits $(t,\delta)$-trees and $f^{opp}$ omits $(t',\delta')$-trees.
Apply Corollary \ref{cor:treeconstant} to $f$ (resp., $f^{opp}$) with parameters $t$ (resp. $t'$), $\delta$ (resp. $\delta'$), and $\epsilon_*=\frac{\epsilon}{9}$ to obtain $s_1\leq (9/\epsilon)^{2t'}$, $s_2\leq (9/\epsilon)^{2t}$, and $f$-definable partitions 
 \[
 X=X'_0\cup X'_1\cup\ldots \cup X'_{s_1}\text{ and }Y=Y'_0\cup Y'_1\cup\ldots \cup Y'_{s_2}
 \]
 satisfying the following properties: 
\begin{enumerate}[$(i)$]
\item  $|X'_0|< \e_*^2|X'_1|$ and $|Y'_0|< \e_*^2 |Y'_1|$. 
\item For each $1\leq i\leq s_1$,  $X'_i$ is $(\delta',2\e_*)$-good with respect to $f$.
\item For each $1\leq j\leq s_2$, $Y'_j$ is $(\delta,2\e_*)$-good with respect to $f$.  
\end{enumerate}

Now set $X_1=X'_0\cup X'_1$ and, for $2\leq i\leq s_1$, set  $X_i=X'_i$. Likewise set $Y_1=Y'_0\cup Y'_1$ and, for $2\leq j\leq s_2$, set $Y_j=Y'_j$. It is easy to check that $X_1$ is $(\delta',2\e_*+\e_*^2)$-good with respect to $f$ and that $Y_1$ is $(\delta,2\e_*+\e_*^2)$-good with respect to $f$. One can also directly check that $2\e_*+\e_*^2\leq \frac{\e}{4}$. So  for all $i\in[s_1]$ and $j\in[s_2]$, $X_i$ is $(\delta',\frac{\e}{4})$-good with respect to $f$ and $Y_j$ is $(\delta,\frac{\e}{4})$-good with respect to $f$. By Lemma \ref{lem:twosticks}, for all $(i,j)\in[s_1]\times[s_2]$, $(X_i,Y_j)$ is $\epsilon$-almost $(\delta+\delta')$-constant with respect to $f$.
\end{proof}

\begin{remark}\label{rem:thm1}
We make several comments on Theorem \ref{thm:maintree}$(a)$.
\begin{enumerate}[$(1)$]
\item  By averaging, the statement of Theorem \ref{thm:maintree}$(a)$ implies that there are $i\in[s_1]$ and $j\in [s_2]$ such that $|X_i|\geq(\epsilon/9)^{2t'}|X|$ and $|Y_j|\geq(\epsilon/9)^{2t}|Y|$. However, from the application of  Corollary \ref{cor:treeconstant}  in the proof, we see that this holds \emph{for all} $i\in [s_1]$ and $j\in[s_2]$. In other words, even though we do not have an equipartition, it is still the case that each set in the partition has positive density, uniformly bounded from below.

\item The assumption of omitting trees in both $f$ and $f^{opp}$ could be consolidated, but potentially at significant cost to the bounds. In particular, \cite[Corollary 1.13]{CT-hodges} shows that if $f\colon X\times Y\to [0,1]$ omits $(t,\delta)$-trees then for all $\e>0$,  $f^{opp}$ omits $(\lceil\e\inv\rceil 2^{2^{t+1}},2\delta+\e)$-trees. On the other hand, in some special cases (e.g., symmetric functions) one could have a much stronger quantitative  relationship between trees forbidden by $f$ versus by $f^{opp}$. 

\item Suppose $f\colon X\times X\to [0,1]$ is a symmetric function omitting $(t,\delta)$-trees. Then in Theorem \ref{thm:maintree}$(a)$, we can take $Y=X$, $\delta=\delta'$, and $t=t'$. In the proof, the two applications of  Corollary \ref{cor:treeconstant} to $f$ and $f^{opp}$ can be identified, which yields $s_1=s_2$ and $X_i=Y_i$ for all $i$. As a corollary, we obtain an analogous symmetric version of Theorem \ref{thm:1c}.  
\end{enumerate}
\end{remark}

Using Corollary \ref{cor:treeconstant2} in place of Corollary \ref{cor:treeconstant}, one can obtain the following variation of Theorem \ref{thm:maintree}$(a)$ with improved bounds, at the cost of small exceptional sets in the partitions. 

\begin{theorem}\label{thm:1alt}
Fix $t,t'\geq 1$ and $\delta,\delta',\e>0$ with $\e<1$.  Suppose $f\colon X\times Y\rightarrow [0,1]$ is such that $f$  omits $(t,\delta)$-trees and $f^{opp}$ omits $(t',\delta')$-trees.    Then there are $s_1\leq (8/\epsilon)^{t'}$, $s_2\leq (8/\epsilon)^{t}$, and $f$-definable  partitions $X=X_0\cup X_1\cup \ldots \cup X_{s_1}$ and $Y=Y_0\cup Y_1\cup \ldots \cup Y_{s_2}$ such that $|X_0|<\e|X|$, $|Y_0|<\e|Y|$ and, for all $(i,j)\in[s_1]\times[s_2]$, $(X_i,Y_j)$ is $\epsilon$-almost $(\delta+\delta')$-constant with respect to $f$.
\end{theorem}

In the previous result, the exceptional sets $X_0$ and $Y_0$ can be removed by distributing them proportionally to the other sets in the partition. However,  this would produce partitions that are no longer necessarily $f$-definable and result in slightly worse bounds due to a required change in the $\epsilon$ parameter.

\section{Equipartitions of weakly good sets}\label{sec:equipartition}

In this section, we begin the work needed to prove Theorem \ref{thm:1ceq}, which is our second version of regularity for stable functions, this time with equipartitions. Our main goal of this section is the following equipartitioning result for weakly good sets which will be crucial for proving Theorem \ref{thm:1ceq}.

\begin{theorem}\label{thm:equi}
Fix $s_1, s_2\geq 1$ and $\delta_1,\delta_2,\e,\lambda_1,\lambda_2,\nu,\rho>0$ with $\rho,\nu\leq\frac{1}{2}$. Let $X$ and $Y$ be sets, with $|X|,|Y|\geq \Omega_{\lambda_1,\lambda_2,\nu}(1)$, and fix a function $f\colon X\times Y\to [0,1]$. Suppose 
\[
X=A_0\cup A_1\cup\ldots\cup A_{s_1}\quad\text{and}\quad Y=B_0\cup B_1\cup\ldots\cup B_{s_2}
\]
are partitions such that  $|A_0|<\rho|X|$, $|B_0|<\rho|Y|$, and, for all $(i,j)\in[s_1]\times[s_2]$, $|A_i|\geq \lambda_1|X|$,  $|B_j|\geq\lambda_2|Y|$, and $(A_i,B_j)$ is weakly $(\delta_1,\delta_2;\e)$-good with respect to $f$.

Set $\tilde{\e}=(1+6\nu)\e+7\nu+2\rho$. Then there are $\ell_1<2(\lambda_1\nu)\inv$, $\ell_2<2(\lambda_2\nu)\inv$, and equipartitions
\[
X=X_1\cup\ldots\cup X_{\ell_1}\quad\text{and}\quad Y=Y_1\cup\ldots\cup Y_{\ell_2}
\]
such that, for all $(i,j)\in [\ell_1]\times [\ell_2]$, $(X_i,Y_j)$ is weakly $(\delta_1,\delta_2;\tilde{\e})$-good with respect to $f$.
\end{theorem}

The proof will be broken down into several steps.
The first key tool  is a tail inequality for hypergeometric distributions, which is also used in Malliaris and Shelah's \cite{MS} approach to the analogous equipartitioning of good sets in graphs.\footnote{See Ackerman, Freer, and Patel \cite{AFP} for a detailed account in the setting of arbitrary relational structures.} The suitable formulation for our function-theoretic setting is  Hoeffding's inequality \cite{hoeffding} for sampling without replacement. We now state this inequality in a form following directly from \cite[Proposition 1.2]{BarMai}. We remind the reader of the notation $(A)_m$ for the set of $m$-tuples from $A$ with pairwise distinct coordinates (see item (\ref{not:distinct}) in Section \ref{notation}). 

\begin{proposition}[Hoeffding's Inequality]\label{prop:hoeff}
Fix $n\geq 1$ and $f\colon A\to[0,1]$, where $A$ is a set of size $n$. Then for any $m\geq 1$ and $\nu>0$,
\[
|\{\xbar\in (A)_m:\E_{\xbar} f\geq \E_{A}f+\nu\}|\leq \exp(-2m\nu^2)|(A)_m|.
\]
\end{proposition}

Using this,  we can obtain the first step toward Theorem \ref{thm:equi}, which says that given finitely many functions on a sufficiently large set $A$, we can find a nontrivial equipartition of $A$ into pieces on  which the averages of the functions do not increase very much. This is a direct analogue of \cite[Proposition 4.4]{AFP}, which deals with the special case of indicator functions.

\begin{lemma}\label{lem:AFP}
Fix $m,q,w\geq 1$ and $\nu>0$ such that $qw<\exp(2\nu^2m)$. Let $A$ be a set of size $wm$. Then for any  functions $f_1,\ldots,f_q\colon A\to[0,1]$, there is a partition $A=A_1\cup\ldots\cup A_w$ such that, for all $i\in[w]$, $|A_i|=m$  and for all $t\in[q]$, $\E_{A_i}f_t< \E_{A}f_t+\nu$.
\end{lemma}
\begin{proof} 
Set $n=wm$. Without loss of generality, assume $A=[n]$. Let $S_n$ be the set of permutations of $A$. For
$\sigma\in S_n$ and $i\in[w]$, set
\[
A_i^\sigma=\{\sigma((i-1)m+1),\ldots,\sigma(im)\}.
\]
Then every $\sigma\in S_n$ determines an equipartition $A=A_1^\sigma\cup\ldots\cup A_w^\sigma$.

Given $i\in[w]$ and $t\in[q]$, define
\[
 X_{i,t} = \{\sigma\in S_n: \E_{A_i^\sigma}f_t\geq \E_{A}f_t+\nu\}.
\]
It suffices to show that for any $i\in[w]$ and $t\in[q]$, $|X_{i,t}|<n!/qw$. Indeed, this implies that there is some $\sigma\in S_n\backslash \bigcup_{(i,t)\in[w]\times[q]} X_{i,t}$, which yields the desired equipartition of $A$. So fix $i\in[w]$ and $t\in[q]$.  Define
\[
Y_t=\{\xbar\in (A)_m:\E_{\xbar} f_t\geq \E_{A}f_t+\nu\}.
\]
Then Proposition \ref{prop:hoeff} implies
\begin{equation}\label{eq:YitHof}
|Y_t|\leq\exp (-2m\nu^2)|(A)_m|<\frac{|(A)_m|}{qw},
\end{equation}
where the second inequality uses our assumption that $qw<\exp(2\nu^2m)$. By an easy counting exercise, we also have
\begin{equation}\label{eq:XitYit}
|X_{i,t}|=(n-m)!|Y_t|=\frac{n!|Y_t|}{|(A)_m|}.
\end{equation}
Together,  (\ref{eq:YitHof}) and  (\ref{eq:XitYit})  yield $|X_{i,t}|<n!/qw$, as desired. 
\end{proof}

In the next step, we extend the previous lemma to the case where $A$ is not exactly divisible by $m$. To streamline the statements going forward, we  adopt the following terminology. 

\begin{definition}
Fix $\e>0$ and a set $X$. Given $\cF\seq [0,1]^X$, we say that a nonempty set $A\seq X$ is \emph{$\e$-small for $\cF$} if for all $f\in\cF$, $\E_A f<\e$.
\end{definition}

\begin{lemma}\label{lem:equi}
Fix $q\geq 1$ and $\e,\mu,\nu>0$ with $\mu<1$. Let $X$ be a set and fix $A\seq X$ with 
\[
|A|\geq 2\mu\inv\quad\text{and}\quad 2q\mu\inv<\exp(2\nu^2\lfloor\mu|A|\rfloor).
\]
Suppose $A$ is $\e$-small for $\cF\seq [0,1]^X$ with $|\cF|\leq q$. Then there is an integer $w\leq 2\mu\inv$ and a partition $A=A_0\cup A_1\cup\ldots\cup A_w$ satisfying the following properties:
\begin{enumerate}[$(i)$]
\item $|A_0|<\mu|A|$.
\item For all $i\in[w]$, $|A_i|=\lfloor\mu|A|\rfloor$ and $A_i$ is $((1-\mu)\inv\e+\nu)$-small for $\cF$.
\end{enumerate}
\end{lemma}
\begin{proof}
Set $m=\lfloor\mu|A|\rfloor$ and write $|A|=wm+s$ for some integers $w>0$ and $0\leq s\leq m-1$. Note that 
\begin{equation}\label{eq:wbound}
w\leq \frac{|A|}{m}\leq \frac{|A|}{\mu|A|-1}\leq 2\mu\inv,
\end{equation}
where the last inequality holds since $|A|\geq 2\mu\inv$. Let $A_0\seq A$ be some subset of size $s$ and let $A'=A\backslash A_0$. Note that $|A_0|<m\leq\mu|A|$, hence we have condition $(i)$. Also,  $|A'|>(1-\mu)|A|$ and, consequently, for every $f\in\cF$,
\begin{equation}\label{eq:A'f}
\E_{A'} f\leq (1-\mu)\inv\E_A f<(1-\mu)\inv\e,
\end{equation}
where the final inequality holds since $A$ is $\e$-small for $f$.

Note that $|A'|=wm$. We also have $qw\leq 2q\mu\inv< \exp(2\nu^2m)$ by (\ref{eq:wbound}) and our initial assumptions. Thus we  may apply Lemma \ref{lem:AFP} to find an equipartition $A'=A_1\cup\ldots\cup A_w$ such that for all $i\in[w]$, $|A_i|=m$ and for all $f\in\cF$, $\E_{A_i}f<\E_{A'} f+\nu$. Combined with (\ref{eq:A'f}), this yields $(ii)$.
\end{proof}

Next we apply the previous lemma simultaneously to the pieces of  an initial partition of a set $X$. This will be the key technical  ingredient for the proof of Theorem \ref{thm:equi}. 

\begin{lemma}\label{lem:equi2}
Fix $q,s\geq 1$ and $\e,\lambda,\nu,\rho>0$ with $\nu<1$. Let $X$ be a set such that
\[
 |X|\geq 2(\lambda\nu)\inv \quad\text{ and }\quad 2q(\lambda\nu)\inv<\exp(2\nu^2\lfloor \lambda\nu |X|\rfloor).
\]
Fix $\cF_1,\ldots,\cF_{s}\seq [0,1]^X$ with $|\cF_i|\leq q$ for all $i\in[s]$. Suppose
\[
X=A_0\cup A_1\cup\ldots\cup A_s
\]
is a partition such that $|A_0|<\rho|X|$ and, for all $i\in[s]$, $|A_i|\geq\lambda|X|$ and $A_i$ is $\e$-small for $\cF_i$. Then there is an integer $\ell<2(\lambda\nu)\inv$ and a partition $X=X_0\cup X_1\cup\ldots\cup X_\ell$
satisfying the following properties:
\begin{enumerate}[$(i)$]
\item $|X_0|<(\rho+\nu)|X|$ and $|X_1|=\ldots=|X_\ell|$.
\item For all $i\in[\ell]$ there is some $i^*\in[s]$ such that $X_i\seq A_{i^*}$ and $X_i$ is $((1-\nu)\inv\e+\nu)$-small for $\cF_{i^*}$.  
\end{enumerate}
\end{lemma}
\begin{proof}
We will first partition each set $A_1,\ldots,A_s$  using Lemma \ref{lem:equi}.  Toward this end, define
\[
\mu=\min\{\nu|A_i|/|X|:1\leq i\leq s\}.
\]
Note that $\mu\geq\lambda\nu$ since $|A_i|\geq\lambda|X|$ for all $i\in[s]$. So $|X|\geq 2\mu\inv$ by assumption. 

Now fix $i\in[s]$ and set $\mu_i=\mu|X|/|A_i|$. Then $\mu_i|A_i|=\mu|X|$, hence $|A_i|\geq 2\mu_i\inv$ since $|X|\geq 2\mu\inv$.
Moreover, since $\mu_i\geq\mu\geq\lambda\nu$, we have 
\[
2q\mu_i\inv\leq 2q(\lambda\nu)\inv<\exp(2\nu^2\lfloor\lambda\nu|X|\rfloor)\leq \exp(2\nu^2\lfloor \mu|X|\rfloor)=\exp(2\nu^2\lfloor \mu_i|A_i|\rfloor).
\]
Since $A_i$ is $\e$-small for $\cF_i$, we can therefore apply Lemma \ref{lem:equi} to $A_i$ with $q$, $\e$, $\mu_i$,  $\nu$, and $\cF_i$.  This yields an integer $w_i\leq 2\mu_{i}\inv$ and a partition 
\[
A_i=A_{i,0}\cup A_{i,1}\cup\ldots\cup A_{i,w_i}
\] 
such that $|A_{i,0}|<\mu_{i}|A_i|$ and, for all $t\in [w_i]$, $|A_{i,t}|=\lfloor \mu_{i}|A_i|\rfloor$ and $A_{i,t}$ is $((1-\mu_i)\inv\e+\nu)$-small for $\cF_i$. Note that 
\[
\mu_{i}=\mu|X|/|A_i|\leq (\nu|A_i|/|X|)(|X|/|A_i|)=\nu.
\]
So, recalling  $\mu_i|A_i|=\mu|X|$, we altogether have the following properties:
\begin{enumerate}[$(a)$]
\item $|A_{i,0}|<\nu|A_i|$.
\item For all $t\in [w_i]$, $|A_{i,t}|=\lfloor \mu|X|\rfloor$ and $A_{i,t}$ is $((1-\nu)\inv\e+\nu)$-small for $\cF_i$. 
\end{enumerate}

Now set $X_0=A_0\cup A_{1,0}\cup\ldots\cup A_{s,0}$. Choose $X_1,\ldots,X_{\ell}$ to be any enumeration of the collection of sets $\{A_{i,t}:1\leq i\leq s,~1\leq t\leq w_i\}$. So we have a partition $X=X_0\cup X_1\cup\ldots\cup X_{\ell}$.  Note that by $(b)$, each set $X_1,\ldots,X_\ell$ has size $m\coloneqq\lfloor \mu|X|\rfloor$ by construction. Therefore 
\[
m+1>\mu|X|\geq\mu\sum_{i=1}^{\ell}|X_i|=\mu\ell m\geq\lambda\nu \ell m,
\]
which implies $\ell<\frac{m+1}{m}(\lambda\nu)\inv\leq  2(\lambda\nu)\inv$. Thus it remains to verify properties $(i)$ and $(ii)$ in the statement.

For property $(i)$, we have
\[
|X_0|\leq|A_0|+\sum_{i=1}^s|A_{i,0}|<\rho|X|+\nu\sum_{i=1}^s |A_i|\leq (\rho+\nu)|X|,
\]
where the second inequality follows from $(a)$ and the assumption $|A_0|<\rho|X|$, and the third inequality uses the fact that $A_1,\ldots,A_s$ are pairwise disjoint.

For property $(ii)$, fix $i\in[\ell]$. By construction, there is some $i^*\in[s]$ and $t\in[w_{i^*}]$ such that $X_i=A_{i^*,t}$. So $X_i\seq A_{i^*}$ and, moreover, $X_i$ is $((1-\nu)\inv\e+\nu)$-small for $\cF_{i^*}$ by $(b)$. 
\end{proof}

The last ingredient  is the following standard exercise.

\begin{remark}\label{rem:distribute}
Fix $\ell\geq 1$ and $0<\rho\leq1$. Let $X$ be a set and suppose $X=X'_0\cup X'_1\cup\ldots\cup X'_{\ell}$ is a partition such that $|X'_0|<\rho|X|$ and $|X'_1|=\ldots=|X'_\ell|$. Let $X=X_1\cup\ldots\cup X_\ell$ be an equipartition obtained by distributing $X'_0$ to $X'_1,\ldots,X'_\ell$ as evenly as possible. Then for all $i\in[\ell]$, $|X_i\backslash X'_i|< \rho|X_i|+1$.
\end{remark}

We now have all the tools necessary to prove Theorem \ref{thm:equi}. However, since the argument is rather elaborate, we will first prove a warm-up result, Proposition \ref{prop:equisym} below. Proposition \ref{prop:equisym} is a one-sided version in the setting of good sets (rather than weakly good sets), and also assumes a more restrictive relationship on the relative sizes of $X$ and $Y$.   This  will not be needed for the proof of Theorem \ref{thm:equi}, but it will be used later in the context of symmetric functions (Theorem \ref{thm:1.5sym}).

\begin{proposition}\label{prop:equisym}
Fix $s\geq 1$ and $\delta,\e,\lambda,\nu,\rho>0$ with $\nu,\rho\leq\frac{1}{2}$. Let $X$ and $Y$ be sets, with $|X|\geq \Omega_{\lambda,\nu}(\log 2|Y|)$,
and fix a function $f\colon X\times Y\to [0,1]$. Suppose 
\[
X=A_0\cup A_1\cup\ldots\cup A_s
\]
is a partition 
such that  $|A_0|<\rho|X|$ and, for all $i\in[s]$, $|A_i|\geq \lambda|X|$ and $A_i$ is  $(\delta,\e)$-good with respect to $f$.

Set $\tilde{\e}=(1+2\nu)\e+3\nu+\rho$. Then there is an integer $\ell<2(\lambda\nu)\inv$ and an  equipartition
\[
X=X_1\cup\ldots\cup X_{\ell}
\]
such that, for all $i\in[\ell]$, $X_i$ is  $(\delta,\tilde{\e})$-good with respect to $f$.
\end{proposition}
\begin{proof}
For each $i\in [s]$ and $y\in Y$, fix $E_{i,y}\seq A_i$ such that $|E_{i,y}|<\e|A_i|$ and $f_y$ is $\delta$-constant on $A_i\backslash E_{i,y}$. Given $i\in[s]$, set $\cF_i=\{\boldsymbol{1}_{E_{i,y}}:y\in Y\}$. Then for any $i\in[s]$, $A_i$ is $\e$-small for $\cF_i$ since, given $y\in Y$, $\E_{A_i}\boldsymbol{1}_{E_{i,y}}=|E_{i,y}|/|A_i|<\e$.

Note that for each $i\in[s]$, $|\cF_i|\leq |Y|$. Since $X$ is sufficiently large in terms of $|Y|$, we can therefore apply Lemma \ref{lem:equi2} to the partition $X=A_0\cup A_1\cup\ldots\cup A_s$ to obtain $\ell<2(\lambda\nu)\inv$ and a partition $X=X'_0\cup X'_1\cup\ldots\cup X'_\ell$ satisfying the following properties:
\begin{enumerate}[$(i)$]
\item $|X'_0|<(\rho+\nu)|X|$ and $|X'_1|=\ldots=|X'_\ell|$.
\item For all $i\in [\ell]$ there is some $\sigma(i)\in[s]$ such that $X'_i\seq A_{\sigma(i)}$ and $X'_i$ is $\e_1$-small for $\cF_{\sigma(i)}$, where $\e_1=((1-\nu)\inv\e+\nu)$. 
\end{enumerate}

Construct an equipartition
\[
X=X_1\cup\ldots\cup X_\ell
\]
by distributing $X'_0$ to $X'_1,\ldots,X'_\ell$ as evenly as possible. We will show that this equipartition is as desired. First, recall that $\ell<2(\lambda\nu)\inv$. Also, by $(i)$ and Remark \ref{rem:distribute}, if $i\in[\ell]$ then
\begin{equation}\label{eq:distoff2}
|X_i\backslash X'_i|<(\rho+\nu)|X_i|+1\leq (\rho+2\nu)|X_i|,
\end{equation}
where the final upper bound holds since $X$ is sufficiently large.

Set $\e_2=\e_1+2\nu+\rho$. Then we have
\[
\e_2=\frac{\e}{1-\nu}+3\nu+\rho=\left(1+\frac{\nu}{1-\nu}\right)\e+3\nu+\rho\leq (1+2\nu)\e+3\nu+\rho=\tilde{\e},
\]
where the final inequality uses $\nu\leq\frac{1}{2}$. So it suffices to show that for all $i\in[\ell]$, $X_i$ is $(\delta,\e_2)$-good with respect to $f$. 

Fix $i\in[\ell]$ and $y\in Y$, and set $i^*=\sigma(i)$. So $X'_i\seq A_{i^*}$. Define
\[
E^*_{i,y}=(E_{\sigma(i),y}\cap X'_i)\cup (X_i\backslash X'_i).
\]
Note that $E^*_{i,y}\seq X_i$. We will use $E^*_{i,y}$ to witness that $X_i$ is $(\delta,\e_2)$-good with respect to $f$.

First,  we have
\begin{multline*}
|E^*_{i,y}|\leq | E_{\sigma(i),y}\cap X'_i|+(\rho+2\nu)|X_i|\\
=|X'_i|\E_{X'_i}\boldsymbol{1}_{E_{\sigma(i),y}}+(\rho+2\nu)|X_i|
<\e_1|X'_i|+(\rho+2\nu)|X_i|\leq\e_2|X_i|,
\end{multline*}
where the first inequality uses the definition of $E^*_{i,y}$ and (\ref{eq:distoff2}), the second inequality holds since $X'_i$ is $\e_1$-small for $\cF_{\sigma(i)}$ (by $(ii)$), and the third inequality uses $|X'_i|\leq|X_i|$ and the definition of $\e_2$.

Finally, fix $y\in Y$. Then 
\[
X_i\backslash E^*_{i,y}\seq X'_i\backslash E_{\sigma(i),y}\seq A_{\sigma(i)}\backslash E_{\sigma(i),y},
\]
and hence $f_y$ is $\delta$-constant on $X_i\backslash E^*_{i,y}$ by choice of $E_{\sigma(i),y}$. 
\end{proof}

Using the previous result, one can deduce an analogue of Theorem \ref{thm:equi} for partitions with fully good pairs in the special case that $|X|$ and $|Y|$ are mutually exponentially bounded. On the other hand, such a result for arbitrary functions does not hold in full generality (see Proposition \ref{prop:goodcounter}). This is the main reason for our focus on weakly good sets.

We  now prove Theorem \ref{thm:equi}. The rough overall structure of the argument is similar to the proof of Proposition \ref{prop:equisym}. However, the details become more intricate since we have to deal with two partitions simultaneously.

\begin{delayedproof}{thm:equi}
For each $(i,j)\in [s_1]\times [s_2]$, fix $E_{i,j}\seq A_i\times B_j$ (resp., $F_{i,j}\seq A_i\times B_j$) witnessing that $A_i$ is weakly $(\delta_1,\e)$-good (resp., $B_j$ is weakly $(\delta_2,\e)$-good) with respect to $f|_{A_i\times B_j}$. In particular, for all  $(i,j)\in[s_1]\times [s_2]$, 
\begin{equation}\label{eq:EFfibers}
|E_{i,j}|<\e|A_i||B_j|\quad\text{and}\quad |F_{i,j}|<\e|A_i||B_j|.
\end{equation}

 We first aim to apply Lemma \ref{lem:equi2} to the partition $X=A_0\cup A_1\cup\ldots\cup A_{s_1}$.  Toward this end, given $(i,j)\in[s_1]\times[s_2]$, define the maps $\alpha_{i,j},\alpha^*_{i,j}\colon X\to [0,1]$ such that, given $x\in X$,
\[
\alpha_{i,j}(x)=\E_{B_j}\boldsymbol{1}_{E_{i,j}}(x,y) \quad\text{and}\quad \alpha^*_{i,j}(x)= \E_{B_j}\boldsymbol{1}_{F_{i,j}}(x,y).
\]
Given $i\in[s_1]$, define  $\calE_{i}=\{\alpha_{i,j}:j\in[s_2]\}\cup\{\alpha^*_{i,j}:j\in[s_2]\}$. \medskip

\noindent\textit{Claim 1.} For all $i\in[s_1]$, $A_i$ is $\e$-small for $\calE_i$.

\noindent\textit{Proof.} Fix $i\in[s_1]$ and $j\in[s_2]$. Then 
\[
\E_{A_i}\alpha_{i,j}=\E_{A_i}\E_{B_j}\boldsymbol{1}_{E_{i,j}}\quad\text{and}\quad \E_{A_i}\alpha^*_{i,j}=\E_{A_i}\E_{B_j}\boldsymbol{1}_{F_{i,j}}.
\]
So by (\ref{eq:EFfibers}), $\E_{A_i}\alpha_{i,j}<\e$ and $\E_{A_i}\alpha^*_{i,j}<\e$.\clqed\medskip

Next, set $k_1=\max\{|\calE_i|:i\in[s_1]\}$. Note that 
\[
|Y|\geq |B_1|+\ldots+|B_{s_2}|\geq s_2\lambda_2|Y|,
\]
and hence $s_2\leq\lambda_2\inv$. Thus for all $i\in [s_1]$, $|\calE_i|\leq 2s_2\leq 2\lambda_2\inv$, hence $k_1\leq 2\lambda_2\inv$. Altogether, by Claim 1 and since $X$ is sufficiently large, we can apply  Lemma \ref{lem:equi2}  to the partition $X=A_0\cup A_1\cup\ldots\cup A_{s_1}$ with $s_1,k_1\geq 1$, $\e,\lambda_1,\nu>0$, and $\calE_1,\ldots,\calE_{s_1}$. This yields an integer $\ell_1<2(\lambda_1\nu)\inv$ and a partition $X=X'_0\cup X'_1\cup\ldots\cup X'_{\ell_1}$ satisfying the following properties:
\begin{enumerate}[$(i)$]
\item $|X'_0|<(\rho+\nu)|X|$ and $|X'_1|=\ldots=|X'_{\ell_1}|$.
\item For all $i\in [\ell_1]$ there is some $\sigma(i)\in [s_1]$ such that $X'_i\seq A_{\sigma(i)}$ and $X'_i$ is $\e_1$-small for $\calE_{\sigma(i)}$, where $\e_1=((1-\nu)\inv\e+\nu)$. 
\end{enumerate}
In particular, given $(i,j)\in[\ell_1]\times [s_2]$, property $(ii)$ implies
\begin{equation}\label{eq:EX'_i}
\E_{X'_i}\E_{B_j}\boldsymbol{1}_{E_{\sigma(i),j}}(x,y)<\e_1\quad\text{and}\quad \E_{X'_i}\E_{B_j}\boldsymbol{1}_{F_{\sigma(i),j}}(x,y)<\e_1.
\end{equation}

We next aim to similarly refine $Y=B_0\cup B_1\cup\ldots\cup B_{s_2}$. Given $(i,j)\in[\ell_1]\times[s_2]$, define the maps $\beta_{i,j},\beta^*_{i,j}\colon Y\to [0,1]$ such that, given $y\in Y$,
\[
\beta_{i,j}(y)=\E_{X'_i}\boldsymbol{1}_{E_{\sigma(i),j}}(x,y)\quad\text{ and}\quad \beta^*_{i,j}(y)=\E_{X'_i}\boldsymbol{1}_{F_{\sigma(i),j}}(x,y).
\]
Given $j\in[s_2]$, define $\cF_j=\{\beta_{i,j}:i\in[\ell_1]\}\cup\{\beta^*_{i,j}:i\in[\ell_1]\}$. Set $k_2=\max\{|\cF_j|:j\in[s_2]\}$, and note that $k_2\leq 2\ell_1<4(\lambda_1\nu)\inv$. Also, if $j\in[s_2]$ then $B_j$ is $\e_1$-small for $\cF_j$ by (\ref{eq:EX'_i}) and Fubini. 
Thus we can apply Lemma \ref{lem:equi2} to the partition $Y=B_0\cup B_1\cup\ldots\cup B_{s_2}$ with $s_2,k_2\geq 1$, $\e_1,\lambda_2,\nu>0$, and $\cF_1,\ldots,\cF_{s_2}$. This yields an integer $\ell_2<2(\lambda_2\nu)\inv$ and a partition $Y=Y'_0\cup Y'_1\cup\ldots \cup Y'_{\ell_2}$ satisfying the following properties. 
\begin{enumerate}
\item[$(iii)$] $|Y'_0|<(\rho+\nu)|Y|$ and $|Y'_1|=\ldots=|Y'_{\ell_2}|$.
\item[$(iv)$] For all $j\in [\ell_2]$ there is some $\tau(j)\in [s_2]$ such that $Y'_j\seq B_{\tau(j)}$ and $Y'_j$ is $\e_2$-small for $\calF_{\tau(j)}$, where $\e_2=((1-\nu)\inv\e_1+\nu)$. 
\end{enumerate}

Consider the two partitions $X=X'_0\cup X'_1\cup\ldots\cup X'_{\ell_1}$ and $Y=Y'_0\cup Y'_1\cup\ldots\cup Y'_{\ell_2}$. Construct equipartitions
\[
X=X_1\cup\ldots\cup X_{\ell_1}\quad\text{and}\quad Y=Y_1\cup\ldots\cup Y_{\ell_2}
\]
by distributing $X'_0$ to $X'_1,\ldots,X'_{\ell_1}$ and $Y'_0$ to $Y'_1,\ldots,Y'_{\ell_2}$ as evenly as possible. We will show that these are the desired partitions in the statement of the theorem.

Toward this end, first recall that $\ell_1<2(\lambda_1\nu)\inv$ and $\ell_2<2(\lambda_2\nu)\inv$. Also, given $i\in[\ell_1]$ and $j\in[\ell_2]$, since $X$ and $Y$ are sufficiently large, we have 
\begin{equation}\label{eq:distoff}
|X_i\backslash X'_i|< (\rho+2\nu)|X_i|\quad\text{and}\quad |Y_j\backslash Y'_j|< (\rho+2\nu)|Y_j|,
\end{equation}
by $(i)$, $(iii)$, and Remark \ref{rem:distribute}. 

Unraveling $\e_2$ and $\e_1$, we have
\[
\e_2=\frac{\e}{(1-\nu)^2}+\frac{\nu}{1-\nu}+\nu=\left(1+\frac{2-\nu}{(1-\nu)^2}\nu\right)\e+\frac{\nu}{1-\nu}+\nu\leq (1+6\nu)\e+3\nu,
\]
where the final inequality uses $\nu\leq\frac{1}{2}$. In particular, we have $\e_3\coloneqq \e_2+4\nu+2\rho\leq\tilde{\e}$.
So to prove the theorem, it suffices to show that for all $(i,j)\in [\ell_1]\times[\ell_2]$, the pair $(X_i,Y_j)$ is weakly $(\delta_1,\delta_2;\e_3)$-good with respect to $f$. 

Fix $(i,j)\in [\ell_1]\times[\ell_2]$, and set $i^*=\sigma(i)$ and $j^*=\tau(j)$. So $X'_i\seq A_{i^*}$ and $Y'_j\seq B_{j^*}$. Define
\[
E^*_{i,j} = \big(E_{i^*,j^*}\cap (X'_i\times Y'_j)\big) \cup \big((X_i\backslash X'_i)\times Y_j\big)\cup \big(X_i\times (Y_j\backslash Y'_j)\big).
\]
Note that $E^*_{i,j}\seq X_i\times Y_j$. In the next claim, we use  $E^*_{i,j}$ to witness that $X_i$ is weakly $(\delta_1,\e_3)$-good with respect to $f|_{X_i\times Y_j}$.

\noindent\textit{Claim 2.}$~$
\begin{enumerate}[$(a)$]
\item $|E^*_{i,j}|<\e_3|X_i||Y_j|$.
\item For all $y\in Y_j$, $f_y$ is $\delta_1$-constant on $X_i\backslash E^*_{i,j}(X_i,y)$.
\end{enumerate}

\noindent\textit{Proof.} Part $(a)$. First, by definition of $E^*_{i,j}$ and  (\ref{eq:distoff}), we have
\begin{equation}\label{eq:E*ij1}
|E^*_{i,j}|\leq |E_{i^*,j^*}\cap (X'_i\times Y'_j)|+(2\rho+4\nu)|X_i||Y_j|.
\end{equation}
Moreover
\begin{equation}\label{eq:E*ij2}
|E_{i^*,j^*}\cap (X'_i\times Y'_j)|=|X'_i||Y'_j|\E_{ Y'_j}\E_{X'_i}\boldsymbol{1}_{E_{i^*,j^*}}=|X'_i||Y'_j|\E_{Y'_j}\beta_{i,j^*}<
 \e_2|X_i||Y_j|,
\end{equation}
where the final inequality uses $|X'_i|\leq |X_i|$, $|Y'_j|\leq |Y_j|$, and the fact that  $Y'_j$ is $\e_2$-small for $\cF_{j^*}$ (by $(iv)$). Combining (\ref{eq:E*ij1}) and (\ref{eq:E*ij2}) yields $|E^*_{i,j}|<\e_3|X_i||Y_j|$ by definition of $\e_3$. 

Part $(b)$. Fix $y\in Y_j$. If $y\not\in Y'_j$ then $E^*_{i,j}(X_i,y)=X_i$ by definition of $E^*_{i,j}$, and the claim is trivial. So assume $y\in Y'_j$, and set $W=X_i\backslash E^*_{i,j}(X_i,y)$. By definition of $E^*_{i,j}$, $W=X'_i\backslash E_{i^*,j^*}(X'_i,y)$. Since $X'_i\seq A_{i^*}$, this implies $W\seq A_{i^*}\backslash E_{i^*,j^*}(A_{i^*},y)$. Since $y\in Y'_j\seq B_{j^*}$, we conclude that $f_y$ is $\delta_1$-constant on $W$ by choice of $E_{i^*,j^*}$.\clqed

Finally, by a symmetric argument, the set 
\[
F^*_{i,j} = \big(F_{i^*,j^*}\cap (X'_i\times Y'_j)\big) \cup \big((X_i\backslash X'_i)\times Y_j\big)\cup \big(X_i\times (Y_j\backslash Y'_j)\big)
\]
witnesses that $Y_j$ is weakly $(\delta_2,\e_3)$-good with respect to $f|_{X_i\times Y_j}$. Therefore, the pair $(X_i,Y_j)$ is weakly $(\delta_1,\delta_2;\e_3)$-good with respect to $f$, as desired.
\end{delayedproof}

\begin{remark}
In the previous proof, the ``$\delta$-constant" aspect of the situation is largely irrelevant to the mechanics of the argument. In particular, given a set $X$, suppose $\calP\seq\bigcup_{A\seq X}[0,1]^A$ is a class of functions closed under arbitrary restrictions.   Given $\e>0$ and $f\colon X\times Y\to[0,1]$, define $A\seq X$ to be \emph{weakly $(\calP,\e)$-good} if there is $E\seq A\times Y$ such that $|E|<\e|A||Y|$ and, for every $y\in Y$, the restriction of $f_y$ to $A\backslash E(A,y)$ is in $\calP$. Thus weak $(\delta,\e)$-goodness is the special case where $\calP$ consists of all $\delta$-constant $[0,1]$-valued functions on subsets of $X$. Given restriction-closed $\calP_1\seq \bigcup_{A\seq X}[0,1]^B$ and $\calP_2\seq\bigcup_{B\seq Y}[0,1]^B$, define weakly $(\calP_1,\calP_2;\e)$-good pairs for a function $f\colon X\times Y\to [0,1]$ in the analogous way.   Then the corresponding version of Theorem \ref{thm:equi} follows by the same proof.
\end{remark}

As discussed in the introduction, Theorem \ref{thm:equi} differs significantly from the analogous approach for graphs in \cite{MS} in that we do not need to assume any level of ``tameness" for the function (e.g., analytic analogues of VC-dimension). The cost of this is that the proof is significantly more involved, and the result only holds for weakly good sets. In Appendix \ref{sec:appendixB}, we will show that under further combinatorial assumptions closely related to bounded VC-dimension, one can  prove an analogue of Theorem \ref{thm:equi} for partitions consisting of fully good sets (see Proposition \ref{prop:equi-cover}). However, the conclusions are numerically weaker than what we are able to obtain via Theorem \ref{thm:equi} (see the discussion following Proposition \ref{prop:equi-cover}).

\section{Stable regularity with equipartitions}\label{sec:mainproof2}

We now prove Theorem \ref{thm:maintree}$(b)$ (regularity with equipartitions for functions omitting trees).

\begin{proof}[\textnormal{\textbf{Proof of Theorem \ref{thm:maintree}$(b)$}}]
Fix  $t,t'\geq 1$ and $\delta,\delta',\e>0$ with $\e<1$.
 Let $X$ and $Y$ be sufficiently large finite sets, and suppose $f\colon X\times Y\to [0,1]$ is  such that $f$ omits $(t,\delta)$-trees and $f^{opp}$ omits $(t',\delta')$-trees. Set $\e_0=\frac{\e}{17}$.
Apply Corollary \ref{cor:treeconstant2}  to $f^{opp}$ (resp., $f$) with parameters $t',\delta',\e_0$ (resp., $t,\delta,\e_0$) to obtain $s_1\leq \e_0^{-t'}$, $s_2\leq \e_0^{-t}$, and partitions
\[
X=A_0\cup A_1\cup \ldots \cup A_{s_1}\text{ and }Y=B_0\cup B_1\cup \ldots \cup B_{s_2}
\]
satisfying the following properties:
\begin{enumerate}[$(i)$]
\item  $|A_0|< \e_0|X|$ and $|B_0|< \e_0|Y|$.
\item For each $i\in[s_1]$, $|A_i|\geq \e_0^{t'}|X|$ and $A_i$ is $(\delta',2\e_0)$-good with respect to $f$.
\item For each $j\in[s_2]$, $|B_j|\geq \e_0^{t}|Y|$ and $B_j$ is $(\delta,2\e_0)$-good with respect to $f$.
\end{enumerate}
Then for each $(i,j)\in [s_1]\times [s_2]$, the pair $(A_i,B_j)$ is  $(\delta',\delta;2\e_0)$-good with respect to $f$. So we can apply Theorem \ref{thm:equi} with parameters $\delta^*_1$, $\delta^*_2$, $\e^*$, $\lambda^*_1$,  $\lambda^*_2$, $\nu^*$, and $\rho^*$ chosen as follows: $\delta^*_1=\delta'$, $\delta^*_2=\delta$, $\e^*=2\e_0$, $\lambda^*_1=\e_0^{t'}$, $\lambda^*_2=\e_0^t$,  $\nu^*=\e/527$, and $\rho^*=\e_0$. This yields $\ell_1<2(\lambda^*_1\nu^*)\inv=62(17/\e)^{t'+1}$, $\ell_2<2(\lambda^*_2\nu^*)\inv=62(17/\e)^{t+1}$, and equipartitions
\[
X=X_1\cup\ldots\cup X_{\ell_1}\quad\text{and}\quad Y=Y_1\cup\ldots\cup Y_{\ell_2}
\]
such that, for all $(i,j)\in [\ell_1]\times [\ell_2]$, $(X_i,Y_j)$ is  weakly $(\delta',\delta;\tilde{\e})$-good with respect to $f$, where $\tilde{\e}=(1+6\nu^*)\e^*+7\nu^*+2\rho^*$. Therefore, by Lemma \ref{lem:twosticks},  for all $(i,j)\in[\ell_1]\times[\ell_2]$, $(X_i,Y_j)$ is $4\widetilde{\e}$-almost $(\delta+\delta')$-constant with respect to $f$. Finally, one can calculate that
\[
4\tilde{\e}=\frac{16\e}{17}+\frac{48\e^2}{8959}+\frac{28\e}{527}<\e.\qedhere
\]
\end{proof}

Recall from Remark \ref{rem:thm1}$(3)$ that in the case of a symmetric function $f\colon X\times X\to [0,1]$, our  regularity results with $f$-definable partitions (Theorem \ref{thm:maintree}$(a)$) automatically produce the same partition of $X$ on both sides. However, this is not the case for Theorem \ref{thm:maintree}$(b)$, which relies on our equipartitioning result for weakly good sets (Theorem \ref{thm:equi}). In particular, the proof of Theorem \ref{thm:equi} is inherently asymmetric, and in general will produce different partitions even when applied to symmetric functions. On the other hand, in the symmetric case we can use the easier equipartitioning result for good sets given by Proposition \ref{prop:equisym}. This yields the following analogue of Theorem \ref{thm:maintree}$(b)$ for symmetric functions.

\begin{theorem}\label{thm:1.5sym}
Fix  $t\geq 1$ and $\delta,\e>0$ with $\e<1$.
Let $X$ be a sufficiently large finite set, and suppose $f\colon X\times X\to [0,1]$ is a symmetric function omitting $(t,\delta)$-trees. 
Then there is an integer $\ell< 28(13/\e)^{t+1}$ and an equipartition $X=X_1\cup \ldots \cup X_{\ell}$ such that, for all $(i,j)\in[\ell]\times[\ell]$, $(X_i,X_j)$ is $\e$-almost $2\delta$-constant with respect to $f$.
\end{theorem}
\begin{proof}
  Set $\e_0=\frac{\e}{13}$.
Apply Corollary \ref{cor:treeconstant2}  to $f$  with parameters $t,\delta,\e_0$  to obtain $s\leq \e_0^{-t}$ and a  partition
\[
X=A_0\cup A_1\cup \ldots \cup A_{s}
\]
such that $|A_0|< \e_0|X|$ and, for each $i\in[s]$, $|A_i|\geq \e_0^{t}|X|$ and $A_i$ is $(\delta,2\e_0)$-good with respect to $f$. Apply Proposition \ref{prop:equisym} with $X=Y$ and with  parameters $\delta^*$,  $\e^*$, $\lambda^*$,  $\nu^*$, and $\rho^*$ chosen as follows: $\delta^*=\delta$, $\e^*=2\e_0$, $\lambda^*=\e_0^{t}$,  $\nu^*=\e/182$, and $\rho^*=\e_0$. This yields $\ell<2(\lambda^*\nu^*)\inv=28(13/\e)^{t+1}$ and an equipartition
\[
X=X_1\cup\ldots\cup X_{\ell}
\]
such that, for all $i\in[\ell]$, $X_i$ is $(\delta,\tilde{\e})$-good with respect to $f$, where $\tilde{\e}=(1+2\nu^*)\e^*+3\nu^*+\rho^*$.

Now, since $f$ is symmetric, it follows that for any $(i,j)\in[\ell]\times[\ell]$, the pair $(X_i,X_j)$ is  $(\delta,\delta;\tilde{\e})$-good with respect to $f$, and thus $4\tilde{\e}$-almost $2\delta$-constant with respect to $f$ by  Lemma \ref{lem:twosticks}. Finally, one can calculate that
\[
4\tilde{\e}=\frac{12\e}{13}+\frac{8\e^2}{1183}+\frac{6\e}{91}<\e.\qedhere
\]
\end{proof}

We also obtain a version in terms of omitting ladders. This follows from Theorem \ref{thm:1.5sym} in exactly the same way that we deduced Theorems \ref{thm:1c} and \ref{thm:1ceq} from Theorem \ref{thm:maintree} in Subsection \ref{sec:maintreestatements}.

\begin{theorem}\label{thm:1.csym}
Fix  $k\geq 1$, $\delta>0$, and $0<\e<1$.
Let $X$ be a sufficiently large finite set, and suppose $f\colon X\times X\to [0,1]$ is a symmetric function omitting $(k,\delta)$-ladders. 
Then there is an integer $\ell< 28(13/\e)^{4^k}$ and an equipartition $X=X_1\cup \ldots \cup X_{\ell}$ such that, for all $(i,j)\in[\ell]\times[\ell]$, $(X_i,X_j)$ is $\e$-almost $4\delta$-constant with respect to $f$.
\end{theorem}

\appendix

\section{Further discussion of Theorem \ref{thm:CPC2}}\label{sec:appendix}

Recall that  Theorem \ref{thm:CPC2} states the non-quantitative regularity lemma for stable functions proved by Chavarria, Conant, and Pillay \cite{CPC}. As previously mentioned, this result does not appear in \cite{CPC} exactly in this form. Instead, the closest comparable result is \cite[Theorem 5.2]{CPC}. The purpose of this appendix is to reconcile the differences between Theorem \ref{thm:CPC2} and \cite[Theorem 5.2]{CPC}. We will  then derive precise quantitative bounds for Theorem \ref{thm:CPC2} and \cite[Theorem 5.2]{CPC}.

In order to state \cite[Theorem 5.2]{CPC}, we need some definitions. First, \cite[Theorem 5.2]{CPC}  is not stated in terms of almost constant pairs, but rather with the following notion of homogeneity.

\begin{definition}[\cite{CPC}]
Fix a function $f\colon X\times Y\to[0,1]$ and subsets $A\seq X$ and $B\seq Y$. Given $\delta,\e>0$, we say $(A,B)$ is \emph{$(\delta,\epsilon)$-homogeneous for $f$} if there are $r,s\in [0,1]$  satisfying the following properties.\footnote{In \cite{CPC}, the subsequent properties are stated with weak inequalities in both $\approx_\delta$ and the occurrences of ``$\e$-almost". This minor change in our conventions does not affect the statement of \cite[Theorem 5.2]{CPC}.}
\begin{enumerate}[$(i)$]
\item For $\epsilon$-almost all $a\in A$, for $\epsilon$-almost all $b\in B$, $f(a,b)\approx_\delta r$, i.e.,
\[
\big|\big\{a\in A:|\{b\in B:f(a,b)\approx_\delta r\}|> (1-\epsilon)|B|\big\}\big|> (1-\epsilon)|A|.
\]
\item For $\epsilon$-almost all $b\in B$, for $\epsilon$-almost all $a\in A$, $f(a,b)\approx_\delta s$, i.e.,
\[
\big|\big\{b\in B:|\{a\in A:f(a,b)\approx_\delta s\}|> (1-\epsilon)|A|\big\}\big|> (1-\epsilon)|B|.
\]
\end{enumerate}
\end{definition}

This notion extends a graph-theoretic analogue found in work of  Malliaris and Pillay \cite{MP}. The definition avoids direct reference to product measures, which simplifies some aspects of the model-theoretic setting used in \cite{MP}.

Next, \cite[Theorem 5.2]{CPC} is  not stated in terms of omitting ladders, but rather with a related configuration that we will call an agnostic ladder.\footnote{This configuration is often used to define local stability  in continuous logic, e.g. \cite[Definition 7.1]{BYU}.}

\begin{definition}\label{def:agnostic}
Given $k\geq 1$ and $\delta>0$, an \emph{agnostic $(k,\delta)$-ladder for $f\colon X\times Y\to [0,1]$} consists of sequences $x_1,\ldots,x_k\in X$ and $y_1,\ldots,y_k\in Y$ such that, for all distinct $i,j\in[k]$, 
\[
|f(x_i,y_j)-f(x_j,y_i)|\geq\delta.
\]
\end{definition}

\begin{remark}\label{rem:agnosticsym}
Note that, unlike ladders, agnostic ladders are ``symmetric" in the sense that if $f\colon X\times Y\to [0,1]$ admits an agnostic $(k,\delta)$-ladder, then so does $f^{opp}$. 
\end{remark}

We can now state \cite[Theorem 5.2]{CPC}.\footnote{We replace the ``$(k,\delta)$-stability" terminology from \cite{CPC} with our equivalent terminology of omitting an agnostic $(k,\delta)$-ladder.}

\begin{theorem}[Chavarria, Conant, \& Pillay \cite{CPC}]\label{thm:CPC2alt}
Given $k\geq 1$ and $\delta,\epsilon>0$,  there is a constant $M$ such that the following holds.  Suppose $f\colon X\times Y\rightarrow [0,1]$ omits agnostic $(k,\delta)$-ladders. Then there are $m_1,m_2\leq M$ and partitions $X=X_1\cup \ldots \cup X_{m_1}$ and $Y=Y_1\cup \ldots \cup Y_{m_2}$ such that  for all $(i,j)\in[m_1]\times[m_2]$, $(X_i,Y_j)$ is $(5\delta+\e,\e)$-homogeneous with respect to $f$.
\end{theorem}

Let us now reconcile the differences between Theorems \ref{thm:CPC2alt} and \ref{thm:CPC2}. First, we have the following quantitative relationship between homogeneous pairs and almost constant pairs.

\begin{proposition}
Fix  $f\colon X\times Y\to[0,1]$, subsets $A\seq X$ and $B\seq Y$, and $\delta,\e>0$.
\begin{enumerate}[$(a)$]
\item If $(A,B)$ is $\epsilon^2$-almost $2\delta$-constant with respect to $f$ then it is $(\delta,\epsilon)$-homogeneous for $f$.
\item If $(A,B)$ is $(\delta,\epsilon)$-homogeneous for $f$ then it is $2\epsilon$-almost $2\delta$-constant with respect to $f$. 
\end{enumerate}
\end{proposition}
\begin{proof}
The main observation is that, in the context of Definition \ref{def:deltaconstant}, a function $\phi\colon V\to [0,1]$ is $\delta$-constant if and only if there is some $r\in[0,1]$ such that for all $v\in V$, $\phi(v)\approx_{\delta/2} r$. Indeed, the converse direction is immediate from the triangle inequality, and for the forward direction, if $\phi$ is $\delta$-constant then choose $r=\frac{1}{2}(\max_{v\in V}\phi(v)+\min_{v\in V}\phi(v))$. 

With the above observation in hand, part $(a)$ of the proposition  follows easily from Markov's inequality, and part $(b)$ is a straightforward exercise.
\end{proof}

The previous result shows that the translation from homogeneous pairs  to  almost constant pairs amounts to a doubling of the ``$\delta$-constant" parameter. Consequently, while Theorem \ref{thm:CPC2} has $10\delta+\e$ in its conclusion,  Theorem \ref{thm:CPC2alt}  has $5\delta+\e$ (note that  the necessary changes to $\epsilon$ can be absorbed without affecting the statements).

Next, we state the relationship between ladders and agnostic ladders. This passes through the notion of a \emph{uniform $(k,\delta)$-ladder}, which is a $(k,\delta)$-ladder in which the values $r_1,\ldots,r_k$ are all equal. Note that a uniform $(k,\delta)$-ladder is an agnostic $(k,\delta)$-ladder and, moreover, uniform $(k,\delta)$-ladders satisfy the same symmetry as in Remark \ref{rem:agnosticsym}. Uniform ladders are studied in detail in our companion paper \cite{CT-hodges}. We list here just the implications we will need. Part $(a)$ is an easy pigeonhole argument (see \cite[Proposition A.1]{CT-hodges}), while part $(b)$ is a standard Ramsey argument (see \cite[Proposition A.2]{CPC} or \cite[Proposition A.2]{CT-hodges}).

\begin{proposition}\label{prop:agnosticladders}
Fix $k\geq 1$ and $\e>0$.
\begin{enumerate}[$(a)$]
\item For any $\delta>0$ and any function $f\colon X\times Y\to [0,1]$, if $f$ admits an $(\lceil\e\inv\rceil(k-1)+1,\delta+\e)$-ladder then $f$ admits a uniform $(k,\delta)$-ladder, and thus admits an agnostic $(k,\delta)$-ladder.
\item There is an integer $k'\geq 1$ such that for any $\delta>0$ and any function $f\colon X\times Y\to [0,1]$, if $f$ admits an agnostic $(k',\delta+\e)$-ladder then $f$ admits a uniform $(k,\delta)$-ladder.
\end{enumerate}
\end{proposition}

Since Theorems \ref{thm:CPC2} and \ref{thm:CPC2alt} are not quantitative, it  follows that omitting ladders versus agnostic ladders makes no difference to the statement.

The third and final difference between Theorems \ref{thm:CPC2} and \ref{thm:CPC2alt} is that the latter only explicitly says ``partitions" rather than ``$f$-definable partitions" or ``equipartitions". We reconcile this with  the following two remarks. 

\begin{remark}\label{rem:complexity}
The proof of Theorem \ref{thm:CPC2alt} directly yields $f$-definable partitions. This is explained in \cite[Remark 5.3]{CPC}, which  also includes a detailed discussion of bounding the ``Boolean complexity" of the $f$-definable partitions involved. To make this precise, one must formalize a notion of complexity, which we will not do. Instead, we will just point out how one might derive this in our results involving $f$-definable partitions. These arguments all stem from Corollary \ref{cor:treeconstant}, were construct an $f$-definable partition of an initial set $Y$ by iterating Lemma \ref{lem:findgood} to find the next $f$-definable set as a subset of some $f$-definable set $Y'\seq Y$. The set $Y'$ begins as $Y$, which of course has bounded complexity, and from the proof of Lemma \ref{lem:findgood} we see that the new set is obtained by adding at most $t$ more inequalities in terms of $f$. Thus the complexity of all of the sets in the partition remains bounded in terms of the initial parameters $t$ and $\epsilon$.  
\end{remark}

\begin{remark}\label{rem:strong}
In \cite[Section 6]{CPC}, the authors obtain a version of Theorem \ref{thm:CPC2alt} with equipartitions using methods significantly different from our random sampling process. In particular, the work in \cite{CPC} actually focuses on an extra strong kind of regularity  in which certain  parameters can be made small as a function of the number of parts. In this regime, partitions can be made equitable using only routine methods.  An account of this approach in the graphs case is given in \cite{CTSRE}. A quantitative version of this extra strong form of regularity for stable graphs was first proved by Terry and Wolf (see Theorem 5.15 and Lemma 5.9 in \cite{TW}). A similar analysis for stable functions is certainly possible, based on a merging of the methods in \cite{TW} with the tools developed above.  However, the bounds in \cite[Theorem 5.15]{TW} are Wowzer-type, and thus this strategy would not yield the polynomial bounds we obtain  here without first obtaining a corresponding improvement in the graphs case. Altogether, a better understanding of optimal bounds for these extra strong stable regularity lemmas remains a separate open question.\footnote{For context, we note that the phrase ``strong regularity" is typically used in the context of a closely related but weaker notion. A strong regularity lemma for arbitrary graphs was first proved by Alon, Fischer, Krivelevich, and Szegedy \cite{AFKS} with Wowzer-type bounds (later shown to be necessary in \cite{KalySh}).}
\end{remark}

Finally, we state the precise quantitative bounds for Theorems \ref{thm:CPC2} and \ref{thm:CPC2alt} that one can obtain from our main results. We begin with  Theorem \ref{thm:CPC2} and the corresponding variation for equipartitions.

\begin{theorem}\label{thm:CPC2qu}
Fix $k\geq 1$ and $\delta,\gamma,\e>0$ with $\e<1$. Suppose $f\colon X\times Y\to [0,1]$ omits $(k,\delta)$-ladders.
\begin{enumerate}[$(1)$]
\item There are $s_1\leq (9/\e)^{4^{\lceil 2\gamma\inv\rceil k}}$, $s_2\leq (9/\e)^{4^k}$, and $f$-definable partitions $X=X_1\cup\ldots\cup X_{s_1}$ and $Y=Y_1\cup\ldots\cup Y_{s_2}$ such that for all $(i,j)\in [s_1]\times [s_2]$, $(X_i,Y_j)$ is $\e$-almost $(4\delta+\gamma)$-constant with respect to $f$.
\item If $X$ and $Y$ are sufficiently large, then there are $\ell_1\leq 62(17/\e)^{4^{\lceil 2\gamma\inv\rceil k}}$, $\ell_2\leq 62(17/\e)^{4^k}$, and equipartitions $X=X_1\cup\ldots\cup X_{\ell_1}$ and $Y=Y_1\cup\ldots\cup Y_{\ell_2}$ such that for all $(i,j)\in [\ell_1]\times [\ell_2]$, $(X_i,Y_j)$ is $\e$-almost $(4\delta+\gamma)$-constant with respect to $f$.
\end{enumerate}
\end{theorem}
\begin{proof}
By Proposition \ref{prop:agnosticladders}$(a)$, $f^{opp}$ omits $(\lceil 2\gamma\inv\rceil k,\delta+\frac{\gamma}{2})$-ladders.  Thus parts $(1)$ and $(2)$ follow, respectively, from direct applications of Theorem \ref{thm:1c} and Theorem \ref{thm:1ceq}.
\end{proof}

Next we derive a corresponding quantitative version of Theorem \ref{thm:CPC2alt}. In analogy to Theorem \ref{thm:CPC2qu}, an argument based on Proposition \ref{prop:agnosticladders}$(a)$ and Theorem \ref{thm:1c}/\ref{thm:1ceq} would produce bounds with a similar double-exponential dependence on the  ``slack parameter" $\gamma$. But we can improve this to a single exponential by instead passing through trees. In particular, we have the following variation of Theorem \ref{thm:hodgesfn} suitable for uniform ladders.

\begin{theorem}[{\cite[Corollary 1.12]{CT-hodges}}]\label{thm:hodgesfnOP}
Given $k\geq 1$ and $\delta,\epsilon>0$, if $f\colon X\times Y\to [0,1]$ admits an $(\lceil \e\inv\rceil4^k,2\delta+\e)$-tree, then $f$ admits a uniform  $(k,\delta)$-ladder, and thus admits an agnostic $(k,\delta)$-ladder.
\end{theorem}

Using this, we  deduce the following quantitative version of Theorem \ref{thm:CPC2alt} (phrased with almost constant pairs) and the  corresponding variation for equipartitions.

\begin{theorem}\label{thm:CPCQ1}
Fix $k\geq 1$ and $\delta,\gamma,\epsilon>0$ with $\e<1$.  Suppose $f\colon X\times Y\rightarrow [0,1]$ omits agnostic $(k,\delta)$-ladders.
\begin{enumerate}[$(1)$]
\item There are $s_1,s_2\leq (9/\e)^{2\lceil 2 \gamma\inv\rceil 4^k}$ and $f$-definable  partitions $X=X_1\cup \ldots \cup X_{s_1}$ and $Y=Y_1\cup \ldots \cup Y_{s_2}$ such that for all $(i,j)\in[s_1]\times[s_2]$, $(X_i,Y_j)$ is $\epsilon$-almost $(4\delta+\gamma)$-constant with respect to $f$.
\item If $X$ and $Y$ are sufficiently large, then there are $\ell_1,\ell_2\leq 62(17/\e)^{\lceil 2\gamma\inv\rceil 4^k+1}$ and equipartitions $X=X_1\cup \ldots \cup X_{\ell_1}$ and $Y=Y_1\cup \ldots \cup Y_{\ell_2}$ such that for all $(i,j)\in[\ell_1]\times[\ell_2]$, $(X_i,Y_j)$ is $\epsilon$-almost $(4\delta+\gamma)$-constant with respect to $f$.
\end{enumerate}
\end{theorem}
\begin{proof}
By Theorem \ref{thm:hodgesfnOP}, $f$ omits  $(\lceil 2\gamma\inv\rceil4^k,2\delta+\frac{\gamma}{2})$-trees. By Remark \ref{rem:agnosticsym}, the same is true of $f^{opp}$.  In light of this, the result follows directly from Theorem \ref{thm:maintree}.
\end{proof}

\section{Equipartitioning good sets with covering numbers}\label{sec:appendixB}

Given a function $f\colon X\times Y\to [0,1]$, Proposition \ref{prop:equisym} shows that a partition of $X$ into good sets for $f$ can be nontrivially converted into an equipartition of good sets, provided $|Y|$ is exponentially bounded by $|X|$. The purpose of this section is to prove a similar result where, rather than restricting $|Y|$, we instead bound certain ``covering numbers" for the fiber family $\cF_f$. First, however, we observe that one cannot prove a  result of this kind without any additional assumptions on $f$. In fact, we establish this even in the discrete case. Given a relation $E\seq X\times Y$ and a set $A\seq X$, we say that $A$ is \emph{$\e$-good} with respect to $E$ if for any $y\in Y$, either $|E(A,y)|<\e|A|$ or $|E(A,y)|>(1-\e)|A|$. Note that this is equivalent to saying that $A$ is $(\delta,\e)$-good with respect to $\boldsymbol{1}_E$, where $0<\delta\leq 1$ is arbitrary.

   \begin{proposition}\label{prop:goodcounter}
   Fix an integer $L\geq 1$ and $0<\e<1$. Let $X$ be a set, with $|X|\geq \Omega_{\e,L}(1)$. Then there is a set $Y$ and a relation  $E\seq X\times Y$ satisfying the following properties.
   \begin{enumerate}[$(1)$]
   \item There is a partition of $X$ into pieces of size at least $\frac{\e}{3}|X|$ such that each piece is  $\e$-good with respect to $E$. 
   \item In any equipartition of $X$ into at most $L$ sets, at least one of the pieces is not $\frac{1}{4}$-good with respect to $E$.
   \end{enumerate}
   \end{proposition}
    \begin{proof}
    Set $s=\lfloor 2\e\inv\rfloor$ and let $A_1,\ldots,A_s$ be pairwise disjoint subsets of $X$ each of size $\lceil \frac{\e}{3} |X|\rceil$. Let $A_{s+1}=X\backslash (A_1\cup\ldots\cup A_s)$. Note that $|A_1\cup\ldots\cup A_s|=
    s\lceil\frac{\e}{3} |X|\rceil\leq \frac{2}{3}|X|+2\e\inv$. Since $X$ is sufficiently large, this implies both $|A_{s+1}|\geq\frac{\e}{3}|X|$ and  $|A_{s+1}|> \frac{1}{4}|X|$. 
    
    Let $Y_1$ be the set of subsets of $X$ obtained as a union over some sub-collection of $A_1,\ldots,A_{s+1}$. Let $Y_2$ be the set of subsets $S\seq A_{s+1}$ with $|S|<\e|A_{s+1}|$. Set $Y=Y_1\cup Y_2$ and let $E\seq X\times Y$ be the membership relation, i.e., $E(x,S)$ holds if and only if $x\in S$. 
    
    We claim that each $A_i$ is $\e$-good with respect to $E$, which establishes $(1)$. Indeed,  fix $i\in[s+1]$ and $S\in Y$. We need to show  $|E(A_i,S)|<\e|A_{i}|$ or $|E(A_i,S)|>(1-\e)|A_{i}|$. Note that if $S\in Y_1$ then either $E(A_i,S)=A_i$ or $E(A_i,S)=\emptyset$. Otherwise, if $S\in Y_2$ then either $E(A_i,S)=\emptyset$, or $i=s+1$ and $E(A_{s+1},S)=S$, hence $|E(A_{s+1},S)|=|S|<\e|A_{s+1}|$. In each case, we obtain the desired conclusion.

    It remains to establish (2). Fix an equipartition $X=X_1\cup\ldots\cup X_\ell$ with $\ell\leq L$. \medskip

    \noindent\textit{Case 1.} There is  $j\in[\ell]$ such that  $|X_j\cap A_i|\leq\frac{3}{4}|X_j|$ for all $i\in[s+1]$.

    Let $B_1,\ldots,B_{s+1}$ be an enumeration of $A_1,\ldots,A_{s+1}$ so that $|B_1\cap X_j|\leq\ldots\leq |B_{s+1}\cap X_j|$. Let $S_0=\emptyset$ and for $i\in [s+1]$, set $S_i=B_1\cup\ldots\cup B_i$. Choose the smallest $i\in[s+1]$ satisfying $|X_j\cap S_i|\geq \frac{1}{4}|X_j|$ (note that $i$ exists since $S_{s+1}=X$). Then
    \begin{equation}\label{eq:XjSi}
    |X_j\cap S_i|=|X_j\cap S_{i-1}|+|X_j\cap B_i|<{\textstyle\frac{1}{4}}|X_j|+|X_j\cap B_i|.
    \end{equation}
    Since $|X_j\cap B_i|\leq \frac{3}{4}|X_j|$ by assumption,  (\ref{eq:XjSi}) implies $|X_j\cap S_i|<|X_j|$. Thus $i\leq s$, and hence $X_j\cap B_i$ and $X_j\cap B_{s+1}$ are disjoint subsets of $X_j$, with $|X_j\cap B_i|\leq |X_j\cap B_{s+1}|$. Therefore $|X_j\cap B_i|\leq\frac{1}{2}|X_j|$, and so (\ref{eq:XjSi}) implies $|X_j\cap S_i|<\frac{3}{4}|X_j|$. Altogether, we have
    \[
    {\textstyle\frac{1}{4}}|X_j|\leq |X_j\cap S_i|<{\textstyle\frac{3}{4}}|X_j|.
    \]
    Since $S_i\in Y_1$ and $E(X_j,S_i)=X_j\cap S_i$, it follows that $X_j$ is not $\frac{1}{4}$-good with respect to $E$.\medskip

    \noindent\textit{Case 2.} For all $j\in[\ell]$ there is some $\sigma(j)\in[s+1]$ such that $|X_j\cap A_{\sigma(j)}|>\frac{3}{4}|X_j|$. 

       Set $Z=\bigcup_{j=1}^\ell X_j\backslash A_{\sigma(j)}$ and note that $|Z|<\frac{1}{4}|X|$ by assumption of this case. Now set $I=\{\sigma(j):j\in[\ell]\}$ and set $I^*=[s+1]\backslash I$. Note that if $i\in I^*$ then, for all $j\in[\ell]$, $X_j\cap A_i\seq X_j\backslash A_{\sigma(j)}$. This implies $\bigcup_{i\in I^*}A_i\seq Z$. Since $|A_i|\geq\frac{\e}{3}|X|$ for all $i\in[s+1]$, we conclude
       \[
       \textstyle\frac{\e}{3}|X||I^*|\leq |\bigcup_{i\in I^*}A_i|\leq |Z|<\frac{1}{4}|X|,
       \]
       and hence $|I^*|<\frac{3}{4}\e\inv$. Therefore $\ell\geq |I|> s+1-\frac{3}{4}\e\inv>2\e\inv-\frac{3}{4}\e\inv=\frac{5}{4}\e\inv$. It follows that for all $j\in [\ell]$, we have 
    \begin{equation}\label{eq:Xjbound}
    |X_j|\leq \textstyle\frac{4}{5}\e|X|+1.
    \end{equation}

    Next, since $\bigcup_{i\in I^*}A_i\seq Z$ and $|Z|<\frac{1}{4}|X|$, it follows that $s+1\in I$ (recall $|A_{s+1}|>\frac{1}{4}|X|$). So we may fix  $j\in[\ell]$ such that $\sigma(j)=s+1$, i.e., $|X_j\cap A_{s+1}|>\frac{3}{4}|X_j|$. Set $r\coloneqq\lceil\frac{1}{4}|X_j|\rceil$. Since $X$ is sufficiently large, we have $r\leq \frac{3}{4}|X_j|$.   Since $|X_j\cap A_{s+1}|>\frac{3}{4}|X_j|$, we may fix some $S\seq X_j\cap A_{s+1}$ of size $r$. Then by (\ref{eq:Xjbound}) and since $X$ is sufficiently large,
    \[
    \textstyle r<\frac{1}{4}|X_j|+1\leq \frac{1}{5}\e|X|+\frac{5}{4}\leq \frac{1}{4}\e|X|< \e|A_{s+1}|,
    \]
    which implies $S\in Y_2$. Moreover, $E(X_j,S)=X_j\cap S=S$, and hence $|E(X_j,S)|=r$. Since $\frac{1}{4}|X_j|\leq r\leq \frac{3}{4}|X_j|$, it follows that $X_j$ is not $\frac{1}{4}$-good with respect to $E$. 
   \end{proof}

Returning to the main goal of this section, we now define covering numbers for function classes (following notation in \cite{ABCH}).

\begin{definition}\label{def:cover}
Let $X$ be a set.
\begin{enumerate}[$(1)$]
\item Given $n\geq 1$ and $\xbar\in X^n$, define the pseudometric $\ell^\infty_{\xbar}$ on $[0,1]^X$ by
\[
\ell^\infty_{\xbar}(f,g)=\max_{1\leq i\leq n}|f(x_i)-g(x_i)|.
\]
\item Given $\e>0$, $\cF\seq [0,1]^X$, $n\geq 1$, and $\xbar\in X^n$, an \emph{$\e$-cover of $\cF$ with respect to $\ell^\infty_{\xbar}$} is a subset $\cF_0\seq\cF$ such that for all $f\in\cF$ there is $f'\in\cF_0$ with $\ell^\infty_{\xbar}(f,f')<\e$. 
\item Given $\e>0$, $\cF\seq [0,1]^X$, $n\geq 1$, and $\xbar\in X^n$, let $N(\epsilon,\mathcal{F},\xbar)$  be the minimal cardinality of an $\e$-cover of $\cF$ with respect to $\ell^\infty_{\xbar}$. 
 \end{enumerate}
\end{definition}

Given a set $X$ and a nonempty subset $A\seq X$, we will use $\ell^\infty_A$ and $N(\e,\cF,A)$ to denote $\ell^\infty_{\bar{a}}$ and  $N(\e,\cF,\bar{a})$ (respectively) where $\bar{a}$ enumerates $A$. This notation is well-defined because these pseudometrics do not depend on the particular enumeration $\bar{a}$ of $A$.

Next we adapt the definition of good sets to arbitrary families of functions.

\begin{definition}
Given $\delta,\epsilon> 0$ and $\cF\seq [0,1]^X$, we say that a nonempty set $A\seq X$ is \textbf{$(\delta,\epsilon)$-good with respect to $\cF$} if for all $f\in\cF$,  $f|_A$ is $\e$-almost $\delta$-constant.
\end{definition}

Note that if $f\colon X\times Y\to [0,1]$ is a function, then a subset $A\seq X$ is $(\delta,\e)$-good with respect to $f$ (in the sense of Definition \ref{def:fullgood}) if and only if it is $(\delta,\epsilon)$-good with respect to the fiber class $\cF_f$ (recall item \ref{not:fibers} in Section \ref{notation}). 

\begin{proposition}\label{prop:equi-cover}
Fix $s,k\geq 1$ and $\delta,\e,\eta,\lambda,\nu>0$, with $\nu\leq\frac{1}{2}$. Let $X$ be a set, with $|X|\geq\Omega_{k,\lambda,\nu}(1)$, and fix $\cF\seq [0,1]^X$.  Suppose there is a partition 
\[
X=A_0\cup A_1\cup\ldots\cup A_s
\]
such that $|A_0|<\nu|X|$ and, for all $i\in[s]$, $|A_i|\geq\lambda|X|$,   $A_i$ is $(\delta,\e)$-good with respect to $\cF$, and $N(\eta,\cF,A_i)\leq k$. 

Set $\tilde{\e}=(1+2\nu)\e+4\nu$. Then there is an integer $\ell<2(\lambda\nu)\inv$ and an equipartition 
\[
X=X_1\cup\ldots\cup X_\ell
\]
such that for all $i\in[\ell]$, $X_i$ is $(\delta+2\eta,\tilde{\e})$-good with respect to $\cF$.
\end{proposition}
\begin{proof}
For each $i\in[s]$,  let $\cF_i\seq \cF$ be an $\eta$-cover of $\cF$ with respect to $\ell^\infty_{A_i}$. By assumption, if $i\in[s]$ then $|\cF_i|\leq k$ and $A_i$ is $(\delta,\e)$-good with respect to $\cF_i$. \medskip

\noindent\textit{Claim 1.} There is an integer $\ell<2(\lambda\nu)\inv$, an equipartition
\[
X=X_1\cup\ldots\cup X_\ell,
\]
and a map $\sigma\colon [\ell]\to [s]$ such that, for all $i\in [\ell]$ and $f\in \cF_{\sigma(i)}$, there is $E_{i,f}\seq X_i$ such that $|E_{i,f}|<\tilde{\e}|X_i|$, $X_i\backslash E_{i,f}\seq A_{\sigma(i)}$, and $f$ is $\delta$-constant on $X_i\backslash E_{i,f}$.

\noindent\textit{Proof.} This follows from an application of Lemma \ref{lem:equi2} and Remark \ref{rem:distribute} nearly identical to the proof of Proposition \ref{prop:equisym}. We leave the details to the reader.\clqed\medskip

Finally, we show that the equipartition in Claim 1 is as desired. Fix $i\in[\ell]$ and $f\in\cF$. We want to show that $f$ is $\tilde{\e}$-almost $(\delta+2\eta)$-constant on $X_i$. By choice of $\cF_{\sigma(i)}$, there is some $f'\in\cF_{\sigma(i)}$ such that $\ell^\infty_{A_{\sigma(i)}}(f,f')<\eta$. By Claim 1, $f'$ is $\delta$-constant on $X_i\backslash E_{i,f'}$, which is contained in $A_{\sigma(i)}$. Since $\ell^\infty_{A_{\sigma(i)}}(f,f')<\eta$, it then follows from the triangle inequality that $f$ is $(\delta+2\eta)$-constant on $X_i\backslash E_{i,f'}$. Since $|E_{i,f'}|<\tilde{\e}|X_i|$ by Claim 1, $f$ is $\tilde{\e}$-almost $(\delta+2\eta)$-constant on $X_i$, as desired.
\end{proof}

Compared to our earlier related results (Theorem \ref{thm:equi} and Proposition \ref{prop:equisym}), the main drawback of Proposition \ref{prop:equi-cover}  is that the goodness conclusion does not yield the same ``constant fluctuation"  $\delta$, but rather the larger value $\delta+2\eta$. This does not cause a problem in the discrete case since the $\delta$ parameter is only sensitive to the difference between $0$ and $1$. However, in the analytic setting,  Proposition \ref{prop:equi-cover} leads to a nontrivial loss, especially in situations where $\eta$ must depend on $\delta$. This is precisely what happens in our regularity results for functions, as we now explain.

The relevance of Proposition \ref{prop:equi-cover} in our setting comes from the fact that stability implies  uniform bounds on covering numbers. In the discrete case, this is known as the Sauer-Shelah lemma, which  holds under the much weaker tameness assumption of bounded VC-dimension. Generalizations  for functions are well-established in the literature on statistical learning theory. We state here a recent result of Aiyer, Mansour, Moran, Shao, and Waknine \cite{AMMSW}, which follows from \cite[Lemma 2]{AMMSW}.\footnote{This deduction requires a triangle inequality adjustment to account for the fact that \cite{AMMSW} uses ``improper" covering numbers (i.e., in Definition \ref{def:cover}(2), $\cF_0$ is not required to be contained in $\cF$).}

\begin{lemma}[\cite{AMMSW}]\label{lem:AMMSW}
Fix $d\geq 1$ and $\delta>0$,  and suppose $\cF\seq [0,1]^X$ has $\delta$-fat-shattering dimension $d$. Then for any $n\geq 2$,  $\xbar\in X^n$, and $0<\gamma<1$, $N(\delta+\gamma,\cF,\xbar)\leq n^{14d\log(4/\gamma)\log n}$.
\end{lemma}

 We refer the reader to \cite{AMMSW} for the definition of fat-shattering dimension. The only point needed for this discussion is that if $f\colon X\times Y\to [0,1]$ omits $(t,\delta)$-trees, then the $\delta$-fat-shattering dimension of the fiber family $\cF_f$ is at most $t-1$. In fact,  omitting $(t,\delta)$-trees is equivalent to a bound of $t-1$ on the  ``sequential" $\delta$-fat-shattering dimension of $\cF_f$ (as defined in \cite{RaSrTe}).  Thus Proposition \ref{prop:equi-cover} and Lemma \ref{lem:AMMSW} together provide an alternate approach toward Theorem \ref{thm:maintree}$(b)$ (regularity with equipartitions for functions omitting trees) which avoids  Theorem \ref{thm:equi} (equipartitions of weakly good sets). However, due to the quantitative drawback in Proposition \ref{prop:equi-cover} discussed above, this would yield pairs that are only $\e$-almost $(3\delta+3\delta'+\gamma)$-constant (rather than $(\delta+\delta')$-constant), where $\gamma$ is an independent ``slack parameter".

 \begin{example}\label{ex:AMMSW}
We describe an example showing that the $\delta+\gamma$ term in Lemma \ref{lem:AMMSW} cannot be improved. Fix $0<\delta<\frac{1}{2}$ and $n\geq 1$. Let $X$ be a set of size $2n$ and set $\calS=\binom{X}{n}$.  Then $|\calS|=\binom{2n}{n}\sim 4^n/\sqrt{\pi n}$, and if $A,B\in\calS$ are distinct then $A\backslash B$ and $B\backslash A$ are both nonempty. Now let $r\colon \calS\to (0,\delta)$ be an arbitrary injective function. Define $\cF=\{f_A:A\in\calS\}\seq [0,1]^X$ where $f_A=r(A)+\delta\boldsymbol{1}_A$. 

We first show that $N(\delta,\cF,X)=|\cF|$. More specifically, we fix distinct $A,B\in\calS$ and show $\ell^\infty_{X}(f_A,f_B)>\delta$. Without loss of generality, assume $r(A)>r(B)$. Choose $x\in A\backslash B$. Then $f_A(x)-f_B(x)=r(A)+\delta-r(B)>\delta$, as desired. 

Next we show that $\cF$ has $\delta$-fat-shattering dimension (at most) $1$. Otherwise, by definition, there are $x_1,x_2\in X$, $s_1,s_2\in [0,1]$, and $A_\sigma\in\calS$ for $\sigma\seq[2]$ such that, given $i\in[2]$ and $\sigma\seq[2]$, if $i\in\sigma$ then $f_{A_\sigma}(x_i)\geq s_i+\delta$, and if $i\not\in\sigma$ then $f_{A_\sigma}(x_i)\leq s_i$. For $i\in[2]$, set $f_i=f_{A_{\{i\}}}$ and $r_i=r(A_{\{i\}})$. Then $r_2\leq f_2(x_1)\leq s_1$ and, similarly, $r_1\leq s_2$. Also $r_1+\delta\geq f_1(x_1)\geq s_1+\delta$, and hence $s_1\leq r_1$. Similarly, $s_2\leq r_2$. Altogether $r_1\leq s_2\leq r_2\leq s_1\leq r_1$, and thus $r_1=r_2$.  But this implies $A_{\{1\}}=A_{\{2\}}$, which is impossible since $f_1(x_1)\geq s_1+\delta$ while $f_2(x_1)\leq s_1$. 
 \end{example}

\begin{remark}
 Lemma \ref{lem:AMMSW} is  related to  earlier work of Alon, Ben-David, Cesa-Bianchi, \& Haussler \cite{ABCH}, who show that $N(2\delta,\cF,\xbar)\leq (n/\delta)^{O(d\log(n/d\delta))}$ when $\cF$ has $\delta$-fat-shattering dimension at most $d$ (see inequality (3) in the proof of \cite[Lemma 3.5]{ABCH}\footnote{This source uses  ``$P_\epsilon$-dimension", which corresponds to $2\e$-fat-shattering dimension as defined in \cite{AMMSW}.}). In the ``stable" setting,  Rakhlin, Sridharan, and Tewari \cite[Corollary 6]{RaSrTe} show that if $\cF$ has sequential $\delta$-fat-shattering dimension at most $t$ then  $N(\e,\cF,\xbar)\leq O((n/\e)^t)$ for any $\e>2\delta$.\footnote{This translation from \cite{RaSrTe} requires the same adjustment for ``improper" covers discussed in the context of Lemma \ref{lem:AMMSW}. Also, the  result actually bounds ``sequential covering numbers" which, in general, are larger than standard covering numbers.}
 \end{remark}

The  results in \cite{AMMSW} also have interesting consequences for analytic regularity  in the context of fat-shattering dimension. This will be the topic of forthcoming work.

\section*{Acknowledgments}

 \subsection*{Humans} The authors thank Aaron Anderson, Tom Waknine, and Julia Wolf for comments on a preliminary draft. We also thank Dhruv Mubayi for helpful conversations about hypergeometric distributions.

\subsection*{AI} ChatGPT was used for proofreading and for finding several  relevant and useful results in the literature. It also made the following mathematical contributions:

\begin{enumerate}[$(1)$]
\item Examples \ref{ex:twosticks} and \ref{ex:AMMSW} were  provided by ChatGPT upon direct request.  Proposition \ref{prop:goodcounter} is a simplification and extension of an example provided by ChatGPT.

\item While reading our original proof of Lemma \ref{lem:findgood}, ChatGPT found an innocuous typo that it interpreted as a ``substantive error". This led it to  simplify that part of the argument  using the separation property now appearing in Proposition \ref{prop:separation}. After reading the argument, we formulated the statement of Proposition \ref{prop:separation} in its current form, which is stronger than what ChatGPT provided, and used it to improve certain  quantitative aspects of Theorem \ref{thm:maintree}$(a)$.

\item  A weaker version of the analytic symmetry lemma (Lemma \ref{lem:twosticks}) was  first obtained by the authors at the Fields Institute in 2021. That version showed that a pair of $(\delta,\e)$-good sets  is $2\e^{1/4}$-almost $6\delta$-constant. After some considerable effort, we eventually produced  a much longer and more technical argument showing that for any $\gamma>0$, a pair of $(\delta,\e)$-good sets is $\frac{1+4\gamma}{\gamma}\e$-almost $(3\delta+\gamma)$-constant. We asked ChatGPT to analyze this argument for further improvements and simplifications. It provided a restructuring of the proof based on a more direct application of Proposition \ref{prop:separation}, which showed that a pair of $(\delta,\e)$-good sets is $6\e$-almost $2\delta$-constant. This argument was streamlined by the authors and rewritten based on further conversations with ChatGPT on other parts of the paper (see the next item).  This resulted in the current version for weakly good pairs.

\item The results in Section \ref{sec:equipartition} on equipartitions of good sets underwent significant revision through extended collaboration with ChatGPT. In the first version of this paper (which was written without any use of AI), we used covering numbers and thus provided quantitatively weaker  statements (as detailed in Appendix \ref{sec:appendixB}). In particular, our  first version of Theorem \ref{thm:maintree}$(b)$ provided polynomial bounds in $1/\gamma\e$, with $\gamma>0$ an independent ``slack parameter", and concluded with $(9\delta+\gamma)$-constant pairs (in the $\delta=\delta'$ case). The changes to Lemma \ref{lem:twosticks} discussed above removed the extra parameter $\gamma$ and improved  $9$ to $6$. This led to conversations with ChatGPT toward the possibility of removing the VC-theoretic covering number arguments entirely and further improving $6$ to  $2$ (in the $\delta=\delta'$ case).\footnote{Recall that by Example \ref{ex:twosticks}, this is the best value attainable by our strategy based on Lemma \ref{lem:twosticks}.}  After extended back-and-forth discussion, we isolated the new definition of ``weakly" good sets, and obtained the equipartitioning result for such pairs given in Theorem \ref{thm:equi}. The overall strategy using hypergeometric tail bounds, as well as the general structuring of the lemmas in Section \ref{sec:equipartition}, were all present in our original draft. The observation in Proposition \ref{prop:equisym} about partitioning good sets without VC-theory when $X$ and $Y$ are close in size was also originally present. On the other hand, the finer details in the proof of Theorem \ref{thm:equi} were largely provided by ChatGPT.

\end{enumerate}

All proofs and examples generated by ChatGPT were carefully checked and thoroughly revised by the authors.

\end{document}